\documentclass[11pt, final]{article}

\usepackage[utf8]{inputenc}
\usepackage{amsmath}
\usepackage{amssymb}
\usepackage{amsthm}
\usepackage{bm}
\usepackage{mathtools}
\usepackage{graphicx}
\usepackage{fullpage}
\usepackage{float}
\usepackage{comment}
\usepackage{xfrac}
\usepackage{xcolor}
\usepackage{booktabs}

\newcommand{\oh}{\sfrac{1}{2}}
\newcommand{\bu}{\mathbf{u}}
\newcommand{\calL}{\mathcal{L}}
\newcommand{\calF}{\mathcal{F}}
\newcommand{\calR}{\mathcal{R}}
\newcommand{\rhob}{\bm{\rho}}
\newcommand{\bR}{\mathbb{R}}
\newcommand{\calD}{\mathcal{D}}
\newcommand{\calN}{\mathcal{N}}
\newcommand{\Lnn}{L_{\mathrm{nn}}}

\newtheorem{Proposition}{Proposition}
\newtheorem{Assumption}{Assumption}
\newtheorem{Definition}{Definition}

\newcommand{\rev}[1]{%
  \begingroup
  #1%
  \endgroup
}

\newcommand{\revv}[1]{%
  \begingroup
  #1%
  \endgroup
}

\newenvironment{revblock}
{\begingroup}
{\endgroup}

\newenvironment{revvblock}
{\begingroup}
{\endgroup}

\title{Subgrid-Scale Parameterization in Burgers’ Equation Using Structure-Preserving Neural Networks and Entropy Variables}

\author{
Aijaz Nazir \thanks{Dept. of Mathematics, University of Houston, Houston, TX 77204,
ORCID:  0000-0002-3494-4101, anazir@cougarnet.uh.edu}, \hspace*{2mm}
Ilya Timofeyev\thanks{Dept. of Mathematics, University of Houston, Houston, TX 77204,
ORCID:  0000-0002-3978-4047, itimofey@cougarnet.uh.edu}, \hspace*{2mm}
}

\begin{document}

\maketitle

\begin{abstract}
We present a machine learning approach for developing subgrid-scale (SGS) parametrizations in coarse simulations of partial differential equations. 
We utilize structure-preserving neural networks and entropy variables to learn subgrid fluxes in coarse simulations of the Burgers' equation.
In particular, we employ a decoupled neural network architecture explicitly separating the subgrid corrections into two distinct components: a conservative Flux Potential network 
and an Eddy Viscosity network. 
We demonstrate that this reduced-order framework maintains high physical fidelity, accurately reproducing the energy spectrum, spatial and temporal correlation functions, and dynamical characteristics of the full-scale system. Furthermore, we show that our approach is robust and applicable to parameters outside the training regime. 
\end{abstract} 

% Keywords

\smallskip


%%%%%%%%%%%%%%%%%%%
% Intro
%%%%%%%%%%%%%%%%%%%
\section{Introduction}
\label{sec:intro}
Developing models for unresolved degrees of freedom in coarse-grained simulations of fluid models has been an active research topic for many decades. The literature on this subject is extensive; some examples from computational fluid dynamics (CFD) include Large Eddy Simulations
(e.g., \cite{sagaut2006LES,lesieur2005LES}), 
subgrid-scale viscosity models 
(e.g., \cite{smagorinsky1963,germano1991dynamic}),
POD-ROM (Proper Orthogonal Decomposition Reduced-Order Models) approach (e.g., \cite{quarteroni2015reduced,benner2015survey,rowley2017}),
tensorial ROM
for parametrized systems \cite{mamonov2025priori,mamonov2024tensorial,mamonov2022interpolatory,Rezwan2026swe,mizan2026reduced}. Various extensions of these techniques have also been proposed, which led to advances in computational methods for fluid dynamics.
\rev{In addition, stochastic mode reduction 
\cite{mtv2} has been used to 
develop effective computational models for coarse variables for the 1D Burgers \cite{dat12} and the 1D shallow-water \cite{zadoacti18} equation.}
Recently, neural networks have been used to develop
parametrizations of subgrid terms (e.g. \cite{lino2023current,rasp2018deep,bolton2019applications,chantry2021opportunities,brenowitz2019spatially,krasnopolsky2005new}) with
many examples demonstrating the validity and effectiveness of this approach. \revv{However, with a
complete theory of neural networks} still in the early stages of development, there is no guarantee of the mathematical properties of ML reduced models, such as stability, convergence, response to external perturbations, etc. Therefore, one of the most fruitful future research directions is to combine the advantages of rigorous mathematics and machine learning.

Hyperbolic systems have a rich mathematical structure that has been extensively explored in the literature.
Analytical advancements for hyperbolic systems have led to a better understanding of nonlinear wave phenomena and advances in computational methods.
One important analytical concept in the theory of hyperbolic systems is the existence of the entropy function $\eta(u)$, where $u$ is the vector of dependent variables in the system (see e.g., \cite{dafermos2005hyberbolic}).
This leads to the definition of entropy variables $v = \nabla_u \eta(u)$ and entropy fluxes that provide additional
analytical structure and an alternative symmetric formulation in the space of entropy variables. This symmetric structure plays an important role in both the theoretical analysis of hyperbolic systems and the design of stable numerical methods.
In particular, the entropy function and entropy variables provide a natural framework for constructing entropy-consistent and entropy-stable discretizations.

Recently, the entropy function and entropy variables have been used to 
develop neural network models to learn 
hyperbolic conservation laws from solution trajectories  \cite{chen2024learning,liu2026parametric}.
In this work, we use a similar approach where we learn the entropy function and construct a neural network model to learn subgrid fluxes in coarse simulations of hyperbolic equations. 
In contrast to \cite{chen2024learning,liu2026parametric}, our focus is not on learning the underlying conservation law, \rev{but on adapting the methodology to learn} the subgrid parametrization for unresolved physical processes in a coarse discretization of hyperbolic conservation laws. Moreover, we focus on long stationary simulations reproducing the energy cascade, the spatial and temporal correlation functions, and individual solutions of the hyperbolic system.
Therefore, our work is related to
CFD techniques mentioned at the beginning of this introduction, since our goal is to represent unresolved subgrid dynamics through effective neural network closures.
\rev{The main contribution of this paper is the formulation and numerical validation of a structure-preserving neural-network closure framework for coarse simulations of hyperbolic conservation laws.}

In this work, we utilize a structure-preserving neural network to model the subgrid processes in coarse simulations of the Burgers' equation. \rev{The main contribution of this paper is to introduce and systematically test a general structure-preserving framework for constructing data-driven subgrid parametrizations for coarse discretizations of hyperbolic conservation laws. The proposed framework combines machine learning with analytical properties of hyperbolic systems, particularly the entropy function, entropy variables, and flux-potential formulation. Although the present study focuses on the one-dimensional Burgers' equation, the objective is not merely to construct a model for this particular equation, but to establish a methodology that can subsequently be extended to more general hyperbolic systems and multidimensional fluid-dynamical problems.}

In particular, we use three neural networks to model different components. First, we use an Input-Convex Neural Network (ICNN) \cite{amos2017input} to learn the entropy function $\eta(u)$; entropy variables can then be obtained by automatic differentiation. Second, we use a feed-forward neural network to learn the flux potential. Third, we use a feed-forward neural network with feature variables to develop a Smagorinsky-type \cite{smagorinsky1963} Eddy-Viscosity approximation. Details are presented in section \ref{sec:nn}.

\rev{An important part of the contribution is the numerical evaluation of the complete framework rather than of its individual neural-network components in isolation. We test whether the resulting closure can simultaneously reproduce individual solution trajectories, long-time statistical properties, and the transfer of energy across resolved scales. We also assess the stability of the proposed coarse-grid discretization and compare its performance with standard numerical and subgrid-scale models.}

We add large-scale forcing to perform long-time stationary simulations and show that our neural network reduced model accurately reproduces the energy spectra and the spatial and temporal correlation structure of the fully-resolved simulations. We also demonstrate that the resulting coarse mesh discretization is stable, accurately reproduces individual solutions, and does not require any additional stabilization (e.g., flux limiters). Our results demonstrate that incorporating the mathematical properties of hyperbolic systems directly into a machine learning approach yields an efficient, stable, and accurate subgrid modeling framework. \rev{Thus, the principal contribution of this work is both the formulation of the entropy-based neural closure framework and its validation across several complementary measures of coarse-model performance.}

\begin{comment}
In this work, we utilize a structure-preserving neural network to model the subgrid processes in coarse simulations of the Burgers' equation. In particular, we use three neural networks to model different components. First, we use an Input-Convex Neural Network (ICNN) \cite{amos2017input} to learn the entropy function $\eta(u)$; entropy variables can then be obtained by automatic differentiation. Second, we use a feed-forward neural network to learn the flux potential. Third, we use a 
feed-forward neural network with feature variables to develop a Smagorinsky-type 
\cite{smagorinsky1963}
Eddy-Viscosity approximation.
Details are presented in section \ref{sec:nn}.
We add large-scale forcing to perform long-time stationary simulations and show that our neural network reduced model accurately reproduces the energy spectra and the spatial and temporal correlation structure 
of the fully-resolved simulations.
We also demonstrate that the resulting coarse mesh discretization is stable, accurately reproduces individual solutions, and does not require any additional stabilization (e.g., flux limiters).
Our results demonstrate that incorporating the mathematical properties of hyperbolic systems directly into a machine learning approach yields an efficient, stable, and accurate subgrid modeling framework. 
\end{comment}

The rest of the paper is organized as follows.
In Section \ref{sec:formulation}, we introduce the Burgers' equation and numerical discretization.
In Section \ref{sec:methodology}, we discuss the problem formulation, including the definition of coarse variables, a brief overview of the entropy function and entropy variables, and a summary of the neural network modeling approach.
Section
\ref{sec:data_nn_training}
presents details about the data preparation, network architecture, and training.
Numerical results are presented in Section \ref{sec:num}. Conclusions are discussed in Section \ref{sec:conc}.

%%%%%%%%%%%%%%%%%%%
% Problem Formulation
%%%%%%%%%%%%%%%%%%%
\section{Problem formulation}
\label{sec:formulation}
\rev{
\subsection{Burgers' Equation with stochastic forcing}
The Burgers' equation \cite{burgers1948mathematical} is one of the well-studied nonlinear partial differential equations, often used as a prototype model for turbulence. We consider the stochastically forced Burgers' equation over a periodic domain $x \in [0,L]$ with $L=2\pi$ given by 
\begin{equation} \label{eq:bur}
\frac{\partial}{\partial t} u + \frac{\partial}{\partial x} f(u) = \rho(x,t) 
\end{equation} 
with the nonlinear flux $f(u) = {u^2}/{2}$ and $u \equiv u(x,t)$.
Here, $\rho$ denotes a large-scale stochastic forcing term, defined as
\begin{equation}
\label{eq:force}
\rho(x,t) = A \sum_{k \in K} \left[
\alpha_k(t)\cos\!\left(kx \right)
+ \beta_k(t)\sin\!\left(kx \right)
\right],
\end{equation}
where $A$ represents the forcing amplitude and $L$ is the length of the spatial domain. 
The coefficients $\alpha_k$ and $\beta_k$ are random variables evolving according to the AR(1)
(autoregressive model of order 1)
process 
\begin{equation}
\label{eq:ab}
\alpha_k(t + \Delta t) = \psi \alpha_k(t) + \sigma \epsilon_{k,1}(t), \quad \beta_k(t + \Delta t) = \psi \beta_k(t) + \sigma \epsilon_{k,2}(t),
\end{equation}
where $0 < \psi = 1 - \gamma\Delta t$ with $\gamma>0$ and $\sigma>0$ are AR(1) parameters, and $\epsilon_{k,i}(t)$, $i=1,2$ are i.i.d. Normal $N(0,\Delta t)$ random variables. In this paper we use $\gamma = 1$ and $\sigma = 1.41$. The forcing magnitude is $A = 1$, unless otherwise specified.
Coefficients $\alpha_k$ and $\beta_k$ are time-correlated Normal $N(0,\sigma^2 / (1 - \psi^2))$ random variables. Equation \eqref{eq:ab} can be viewed as a temporal discretization of the Ornstein-Uhlenbeck process. 
In this paper, we use \( K = \{1,2,3\} \), i.e., the first three wavenumbers are forced.
The stochastic forcing $\rho(x,t)$ is introduced to counteract the numerical dissipation inherent in finite-volume schemes.

%%%%%%%%%%%%%%%%%%%
% Spec-Time discretization
%%%%%%%%%%%%%%%%%%%
 \subsection{Space--time discretization}
Next, we consider a finite-volume discretization of the PDE in \eqref{eq:bur}. 
The spatial domain is discretized using a uniform finite-volume mesh with grid spacing $\Delta x = L/N_f$. The computational cells are defined as
\(
C_i = [x_{i-\oh},\, x_{i+\oh}], \qquad i = 0, \ldots, N_f - 1,
\)
where the cell interfaces are located at $x_{i-\oh} = i\Delta x$ and $x_{i+\oh} = (i+1)\Delta x$. The cell centers are given by $x_i = (i + \oh)\Delta x$.
This defines the fine-scale computational mesh
\(
\mathcal{M}_f = \{x_i = i\Delta x \mid i = 0,1,\ldots,N_f-1\}.
\)
The semi-discrete system becomes
\begin{equation} \label{eq:burd}
\frac{\mathrm{d}}{\mathrm{d}t} u_i
=
-\,\frac{f_{i+\oh} - f_{i-\oh}}{\Delta x}
+ \rho_i,
\end{equation}
where $\Delta x = L/N_f$, $u_i$ is the average over the \revv{cell} $C_i$, and $f_{i+\oh}$ is a suitable discretization of the flux function. 
There are various options for the numerical flux $f_{i + \oh}$ with particular details found in standard literature on finite volume and finite difference methods for conservation laws (e.g., \cite{leveque1992numerical,hesthaven2017numerical}). In this paper, we adopt the Local Lax-Friedrichs flux (LLF) \cite{rusanov1961calculation} given by
\begin{equation}
\label{fluxf}
    f_{i + \oh} = \frac{f(u_{i + 1}) + f(u_i)}{2} - \frac{\lambda_{i+\oh}}{2}(u_{i + 1} - u_{i}),
\end{equation}
where $\lambda_{i+\oh}$ is the local bound for the maximum wave speed. 
In practice, $\lambda_{i+\oh}$ is often approximated using the Rusanov formula
\(
    \lambda_{i+\oh} = \max(|f'(u_i)|, |f'(u_{i+1})|) = \max(|u_i|, |u_{i+1}|).
\)

For the temporal discretization, we employ Heun’s method, a second-order explicit strong stability preserving (SSP) \cite{gottlieb2011strong} Runge–Kutta scheme, to advance the semi-discrete system in time. 
Let $\bu = \{u_i(t), i=0,\ldots, N_f-1\}$. Then the semi-discrete system in 
\eqref{eq:burd} can be schematically written as
$\dot{\bu} = g(\bu) + \rhob(t)$.  
The Heun's method can be written as
\begin{equation*}
     \tilde{\bu}^{n+1} = \bu^n + g(\bu^n) \Delta t, \qquad
     \bu^{n+1} = \bu^n + \frac{1}{2} \left( g(\bu^n) + g(\tilde{\bu}^{n+1}) \right) \Delta t + \rhob^n,
\end{equation*}
where $\bu^n$ denotes the numerical solution at time $t_n$.
We also include stochastic forcing 
$\rhob^n = \{\rho_i^n, i=0,\ldots, N_f-1\}$ with
$\rho_i^n \equiv \rho(x_i,t^n)$
into the second step.}

%%%%%%%%%%%%%%%%%%%
% Methodology
%%%%%%%%%%%%%%%%%%%
\section{Methodology}
\label{sec:methodology}
Fully resolved simulations are often prohibitively expensive, making it desirable to obtain a coarse description of the physical phenomena. To this end, it is necessary to develop an appropriate subgrid model that represents the interactions between the resolved and unresolved (subgrid) degrees of freedom. 
\rev{In this work, we formulate such a model by adapting an entropy-based, structure-preserving neural-network methodology to the problem of subgrid closure. The proposed framework combines three neural networks that learn the entropy function, the flux potential, and the eddy-viscosity component, respectively, thus incorporating analytical structure directly into the subgrid parametrization.}

%%%%%%%%%%%%%%%%%%%
% coarse variables
%%%%%%%%%%%%%%%%%%%
\subsection{Coarse Variables and Deviations}
We introduce an averaging operator and define coarse variables as local averages over $q$ fine grid points, i.e., 
\begin{equation}  
\label{eq:U}
    {U_{I}}(t) = \frac{1}{q} \sum_{i = qI}^{q(I+1) -1} u_i(t)
\end{equation}
for each cell $I$ in a coarse mesh $\mathcal{M}_c = \{0, 1, 2, \dots, N_c-1\}$ with $N_c = N_f/q$. The unresolved degrees of freedom are then defined as the residuals or deviations $y_i = u_i - U_{I(i)}$, where the index $i$ for $y_i$ and $u_i$ refers to the coarse cell $I$ where the fine cell $i$ is located. 

The dynamics of $U_I$ can be obtained by averaging the fine-scale equation \eqref{eq:burd}:
\begin{equation}
\label{burc}
    \frac{d}{dt} U_{I} = -\frac{1}{q \Delta x} \sum_{i = qI}^{q(I+1) - 1}  \left( f_{i + \oh} - f_{i - \oh} \right) + \frac{1}{q} \sum_{i = qI}^{q(I+1) - 1} \rho_i(t).  
\end{equation}
Using the telescoping property, the sum in the equation above can be rewritten as 
\[
\sum_{i = qI}^{q(I+1) - 1}  \left( f_{i + \oh} - f_{i - \oh} \right) =
f_{q(I+ 1) - \oh} - f_{qI - \oh}.
\]
Next, since our goal is to remain within the flux-difference framework, we define the ``True'' coarse fluxes as
\begin{equation}
\label{eq:FT}
    F_{I + \oh}^{T} = f_{q(I+ 1) - \oh}, \qquad
    F_{I - \oh}^{T} =  f_{qI - \oh},
\end{equation}
where the superscript $T$ stands for the "True" flux. In addition, we also define $\rho^{U}_I (t)\equiv \frac{1}{q} \sum_{i = qI}^{q(I+1) - 1} \rho_i(t)$ 
as the average of the forcing term on the coarse mesh. 
The equation \eqref{burc} is exact, but not closed, because fluxes 
$f_{i \pm \oh}$ depend on small-scale variables.
Note that $f_{q(I+ 1) -1/2}$ represents the exact physical flux evaluated at the right boundary of the coarse cell $I$. Because it acts on the fine grid interface, it is a function of both the coarse variables and deviations. 
Substituting $u_i = y_i + U_{I(i)}$ into \eqref{fluxf} and computing explicitly the coarse index $I(i)$, the "True" flux
can be decomposed into a macroscopic component and a subgrid flux:
\begin{equation} \label{burf}
    F_{I + \oh}^{T} = F_{I + \oh}(\mathbf{U}) + G_{I + \oh}(\mathbf{U},\mathbf{y}),
\end{equation}
where
\begin{equation} \label{eq:macroflux}
    F_{I + \oh} = 
    \frac{F(U_I) + F(U_{I+1})}{2} = 
    \frac{1}{4} \left ( U^2_{I + 1} + U^2_{I} \right).
\end{equation} 

The subgrid flux $G_{I + \oh}(\mathbf{U},\mathbf{y})$ depends on both the coarse variables and deviations. The goal of subgrid modeling is to develop a suitable approximation of the subgrid flux 
expressed only in terms of the coarse variables, thus eliminating explicit dependence on the deviations.
In this paper, we use structure-preserving neural networks to approximate the subgrid flux $G_{I + \oh}$.

We assume that the forcing is slowly varying in space, and the averaged forcing $\rho^U_I(t)$
can be well-approximated by the forcing evaluated at the mid-point of cell $I$, i.e., 
$\rho^U_I(t) \approx \rho(X_I, t)$. Thus, 
we denote $\rho_I(t) := \rho(X_I, t)$ and omit the subscript $U$ for the rest of the paper.

%%%%%%%%%%%%%%%%%%%
% entropy variables
%%%%%%%%%%%%%%%%%%%
\subsection{Entropy Variables and Symmetrization}
\label{sec:entropy_theory}

Data-driven subgrid models consistent with the underlying physical properties of the corresponding partial differential equations would have a better chance of reproducing the properties of fully-resolved dynamics. Therefore, we design our neural network for estimating subgrid fluxes to be consistent with the theory of entropy-stable conservation laws (see e.g., \cite{dafermos2005hyberbolic, godlewski2013numerical, tadmor1987numerical,tadmor2003entropy}), which we summarize briefly in this section. 

Consider the conservation law
\begin{equation} \label{eq:bur2}
\frac{\partial}{\partial t} u + \frac{\partial}{\partial x} f(u) = 0.
\end{equation}
Many conservation laws admit a strictly convex entropy function $\eta(u)$. The existence of such a function implies that the system is symmetrizable. Moreover, the strict convexity of $\eta$ ensures a one-to-one mapping between the conservative variables u and the entropy variables, defined as:
\begin{equation}
    v = \eta'(u).
\end{equation}
Since $\eta'(u)$ is strictly convex, 
then $v = \eta'(u)$ is invertible, and we can define $u(v)$.
Transforming the governing equations into the dual space of entropy variables symmetrizes the system, since it defines a one-to-one change of variables (at least locally). 
The derivative $u'(v)$ is positive definite since the entropy function $\eta(u)$ is strictly convex.
Next, we can write fluxes in terms of entropy variables 
$g(v) := f(u(v))$ and recast the conservation law in the symmetric form using entropy variables
\[
u'(v) \frac{\partial v}{\partial t} + g'(v) \frac{\partial v}{\partial x} = 0.
\]
We can also define the flux potential 
$\phi(u) = vf(u) - r(u)$ with $r'(u) = \eta'(u) f'(u)$. If we substitute the entropy variables into the flux potential, then we can rewrite the conservation law \eqref{eq:bur2} as
\[ 
\frac{\partial}{\partial t} u + \frac{\partial}{\partial x} \phi'(v) = 0.
\]

%%%%%%%%%%%%%%%%%%%
% NN Modeling
%%%%%%%%%%%%%%%%%%%
\subsection{Neural Network Subgrid Modeling}

In this work, we propose a hybrid formulation for the discrete dynamics of the coarse variable $U_{I}$. Rather than utilizing a black-box neural network to approximate the right-hand side of the equation for coarse variables, we keep the flux formulation and decompose fluxes into three parts: a macroscopic component, a learned structurally conservative part, and a learned dissipative correction. We utilize the following function form for the reduced equations for coarse variables 
\begin{equation}
\label{eq:reduced}
    \frac{d}{dt} U_{I} = -\frac{1}{q \Delta x} \left( \calF_{I+\oh} - \calF_{I-\oh} \right) + \rho_I(t),
\end{equation}
where the numerical flux at the interface $I+1/2$ is constructed to preserve the underlying physical structure of the conservation law:
\begin{equation}
\label{eq:effflux}
    \calF_{I+\oh} = F_{I + \oh} + F^{NN}_{I+\oh}(U_I, U_{I+1}) - \mathcal{D}^{NN}_{I+\oh}(U_I, U_{I+1}).
\end{equation}
Here, $F_{I + \oh}$ is given by equation \eqref{eq:macroflux} and the construction of $F^{NN}_{I+\oh}$ and $\mathcal{D}^{NN}_{I+\oh}$ is described below.

\textbf{Neural Subgrid Flux $F^{NN}_{I+\oh}$ and Entropy Encoding.} 
The structurally conservative part is estimated 
by learning two separate neural networks: the 
\emph{Entropy Neural Network} $\eta_\theta$ and the \emph{Flux Potential Neural Network} $\phi_\theta$. 
To ensure consistency, this term is calculated in the dual space of entropy variables $v = \eta'(U)$ \cite{chen2024learning,liu2026parametric}.
In particular, the flux 
$F^{NN}_{I+\oh}$ is computed as
\[
F^{NN}_{I+\oh} = \frac12 \left[ \phi'_\theta(v_I) + \phi'_\theta(v_{I+1}) \right], \quad 
v_I = \eta'_\theta(U_I),
\]
where derivatives are computed using automated differentiation and $\theta$ denotes the network's parameters.
To mathematically guarantee that $\eta_\theta(U)$ is a strictly convex entropy function, it is parameterized using an Input-Convex Neural Network (ICNN) \cite{amos2017input}. Strict convexity (i.e., $\eta'' > 0$) is architecturally enforced by restricting the hidden-layer weights to the non-negative domain and employing convex, non-decreasing activation functions. Additionally, a quadratic skip-connection is incorporated to maintain a strictly positive lower bound on the second derivative. 
Particular details about the architecture of $\eta_\theta(u)$ and $\phi_\theta(v)$
are provided in sections \ref{sec:icnn} and \ref{sec:fluxnet}, respectively.

\textbf{Dissipation $\mathcal{D}_{NN}$.} 
The idea is similar to the Smagorinsky-type approximation \cite{smagorinsky1963}. Here, a neural network is used to learn the state-dependent viscosity coefficient. In particular, we use the following form for the dissipation term
\begin{equation}
\label{DNN}
    \mathcal{D}^{NN}_{I+\oh} = C_{\theta}(\bm{\xi}) 
    \frac{\Lambda_{I+1/2}}{2} \times (U_{I+1} - U_I).
\end{equation}
Here, $\Lambda_{I+1/2}$ denotes the local bound for the maximum wave speed, defined as $\Lambda_{I+1/2} = \max(|U_I|, |U_{I+1}|)$. The coefficient $C_{\theta}$ is predicted by an \textit{Eddy Viscosity Neural Network}, which interrogates the local flow state through a feature vector $\bm{\xi}$ comprising the local average velocity and the amplified interfacial jump.  
We ensure that $C_{\theta}(\bm{\xi})$ is positive
by applying the sigmoid function in the final layer.
\rev{The sigmoid activation allows us to impose hard bounds on the viscosity coefficient, so that $C_{\theta}(\bm{\xi}) \in [C_{\min}, C_{\max}]$.}

Unlike traditional schemes where the dissipation coefficient is a fixed constant, $C_{\theta}$ is a state-dependent parameter learned through end-to-end backpropagation. This allows the model to perform \textit{Residual Discovery}, where the neural branches ($F^{NN}_{I+\oh}$ and $C_{\theta}$) collaborate to resolve the high-frequency discrepancies between the coarse macroscopic discretization and the filtered high-fidelity Direct Numerical Simulation (DNS) reference. Details about the architecture of $C_\theta$ are provided in section \ref{sec:eddynet}.

%%%%%%%%%%%%%%%%%%%
% Data, Network, and Training
%%%%%%%%%%%%%%%%%%%
\section{Data Preparation, Network Architecture, and Training}
\label{sec:data_nn_training}

We develop a reduced model for the local averages of the \revv{stochastically} forced Burgers' equation by introducing a structure-preserving hybrid ''Gray-Box'' framework. Rather than relying on standard feed-forward neural networks to approximate the entire non-linear flux, we decompose the numerical flux into an analytical macroscopic component and learned structural and dissipative subgrid corrections. We note that standard mean-squared error formulations struggled to accurately capture shock dynamics without over-dissipating smooth waves. \revv{To address this limitation, we introduce a gradient-weighted loss function that emphasizes high-frequency interfacial jumps and train the model end-to-end using the total physical flux.} Furthermore, the spurious oscillations that traditionally require flux limiters are natively suppressed through a state-dependent Eddy Viscosity Network, while consistency for nonlinear flux corrections is architecturally guaranteed via an Input-Convex Neural Network (ICNN). 

The performance of the reduced model is assessed based on how well equilibrium statistical properties are recovered compared to those of the Direct Numerical Simulation (DNS). In particular, we illustrate that the reduced model is able to accurately reproduce the spectral energy decay and the spatial and temporal correlation functions of the large-scale variables. The response of the reduced model with respect to changes in the amplitude of the forcing is also analyzed. 

In this section, we describe how the training dataset is generated, the structure of the three-branch neural architecture, and the latent learning procedure used to approximate the subgrid physics and resolve the dynamics on the coarse mesh.

%%%%%%%%%%%%%%%%%%%
% Data preparation
%%%%%%%%%%%%%%%%%%%
\subsection{Data Preparation}

We conduct a high-resolution numerical simulation of the full model \eqref{eq:burd}, sampling the solution at intervals of $100 \Delta t$ with $\Delta t = 0.001$. Thus, each run generates $10000$ snapshots. To ensure stability and remove transient effects, the initial $5000$ model time units are discarded. The initial velocity field is defined as a superposition of the first two Fourier modes,
\begin{equation*}
u(x,0) =
0.1 \sin\!\left(\frac{2\pi x}{L} + \phi_1\right)
+ 0.1 \sin\!\left(\frac{4\pi x}{L} + \phi_2\right),
\end{equation*}
where $L = 2\pi$ and the phase shifts $\phi_1, \phi_2 \sim \mathcal{N}(0,1)$ are independently sampled from a normal distribution. Time series of variables $U_0$ and $U_1$, the corresponding "true" flux \( F_{\oh}^T \), and the smoothness indicator $\beta_{\oh}$ (computed for coarse variables)
are recorded. We generate three datasets for three different initial conditions, 
filter them as described below, and then 
combine them into one balanced dataset. 
The final dataset
contains approximately $152K$ samples. \rev{Details about training and validation samples are provided in Table \ref{tab:hyperparameters}. }

\textbf{Dataset Filtering.}
Smoothness indicators, introduced in WENO schemes~\cite{jiang1996efficient}, 
provide a quantitative measure of local solution regularity. 
For Burgers' dynamics, the indicator $\beta$ identifies regions of nonlinear 
steepening and shock formation. Thus, \revv{it enables} targeted selection of dynamically active zones for model training. In this work, local smoothness is quantified using a standard WENO-type smoothness indicator. For a discrete solution $U_I$, the indicator $\beta_I$ is defined as
\begin{equation}
\beta_I
= \frac{13}{12}\left(U_{I-1} - 2U_I + U_{I+1}\right)^2
+ \frac{1}{4}\left(U_{I-1} - U_{I+1}\right)^2 .
\end{equation}
The first term approximates the local curvature of the solution through a second-order finite difference, while the second term measures the squared first-order gradient. In several previous works \cite{alcala2021subgrid,mojamder2026subgrid,timofeyev2026subgrid}, we determined that near-shock regions are most challenging for subgrid modeling.
Therefore, to emphasize regions with significant subgrid activity, the dataset is filtered using the smoothness indicator \( \beta \), which quantifies the local variability of the solution. A threshold of \( \beta^* = 0.25 \) is applied, separating the data into high-\( \beta \) (active subgrid regions) and low-\( \beta \) (smoother regions) subsets. Samples with $\beta \ge \beta^*$, which correspond to shocks, are retained in full. We remove $70\%$ of the low-$\beta$ samples from the dataset. This selective downsampling prevents the over-representation of low-$\beta$ values, preserving diversity in the dataset while ensuring sufficient representation of higher-$\beta$ values. This approach ensures that the training dataset effectively captures different behaviors, allowing the neural network to generalize well across different conditions.

%%%%%%%%%%%%%%%%%%%
% NN and Training
%%%%%%%%%%%%%%%%%%%
\subsection{Network Architecture and Training}
\label{sec:nn}
We employ a structure-preserving hybrid neural architecture to learn the subgrid structural correction and dynamic dissipation as functions of the local state, which differs from the usual strategy of learning global solutions directly. The goal is to achieve a local, physics-informed approximation, where the total numerical flux at a specific interface $I + 1/2$ is determined only by the adjacent cell averages $U_I$ and $U_{I+1}$. Thus, our flux representation relies on a compact two-point stencil; this is possible by using the entropy-consistent framework. Due to spatial homogeneity, it is sufficient to use data from a representative local interface (at $\oh$ in our case), allowing for a smaller dataset and efficient parallel training. We then utilize 
the same neural network to compute subgrid fluxes $\calF_{I+\oh}$ for $I=0,\ldots,N_c-1$ in simulations of the reduced dynamics \eqref{eq:reduced}.

The model architecture is implemented in JAX/Flax \cite{jax2018github} and consists of three neural networks, as shown schematically in Figure \ref{fig:greybox_architecture}. First, the \emph{Entropy Neural Network} parameterizes the convex entropy function $\eta$ using an Input-Convex Neural Network. It utilizes \textit{Softplus} activations, non-negative weight constraints enforced via absolute value operations, and a quadratic skip-connection to guarantee a strictly positive-definite Hessian. Second, the \emph{Flux Potential Neural Network} utilizes smoothly differentiable \textit{Tanh} activations to map the dual entropy variables to the flux potential. Finally, the \emph{Eddy Viscosity Neural Network} employs \textit{Swish} activation \cite{ramachandran2017searching} and processes an augmented feature vector containing the local average velocity and an amplified interfacial jump, i.e.,
\[
\bm{\xi} = 
\left[
\frac{U_I + U_{I+1}}{2}, \\
10|U_I - U_{I+1}|
\right].
\]
Its final layer applies a scaled \textit{Sigmoid} function to restrict the dissipation coefficient $C_{\theta}(\bm{\xi})$ strictly within $[C_{\min},C_{\max}] = [0.35, 2.0]$.
% about the range
\rev{Analysis of the neural network numerical scheme in Section \ref{sec:analyt} provides analytical guidelines for selecting constants $C_{\min}$ and $C_{\max}$.}

\begin{figure}[H]
    \centering
    \includegraphics[width=0.8\textwidth]{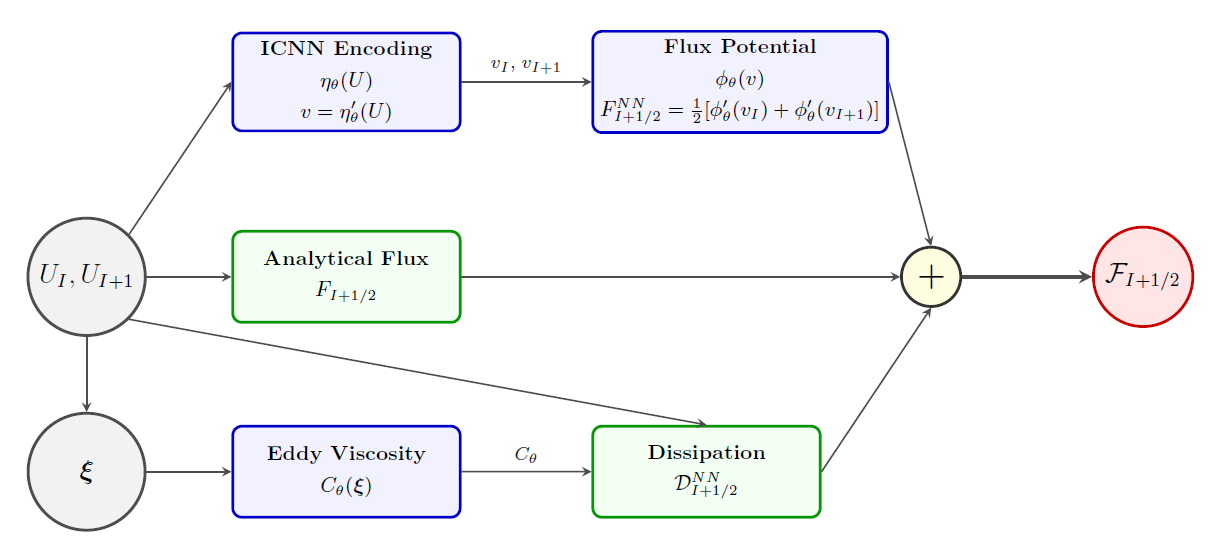} 
    \caption{Schematic representation of the three branches for the structure-preserving ''Gray-Box'' computational pipeline. 
    The flux $F_{I+\oh}$ (middle) is evaluated using formula \eqref{eq:macroflux}. Structural corrections (top) are restricted by learning a convex network $\eta_\theta(U)$
    and the flux potential $\phi_\theta(v)$.
    State-dependent dissipation (bottom) is implemented by learning the viscosity coefficient $C_\theta(\bm{\xi})$ and multiplying it by the LLF-type term (equation \eqref{DNN}). Blue boxes are modeled using different neural networks. Green boxes are given by explicit formulas.}
    \label{fig:greybox_architecture}
\end{figure}

 To train this hybrid architecture, the model is optimized directly on the high-fidelity DNS dataset. Rather than relying on standard mean squared error, which frequently struggles to balance shock capturing with smooth wave preservation, we introduce a \emph{Gradient-Weighted Loss} function
 \[
 \calL = \sum\limits_{\text{batch}} w_{I+\oh} \, ||F_{I+\oh}^T - \calF_{I+\oh}||_2^2 \quad \text{with} \quad
 w_{I+\oh} = 1.0 + \gamma \, |U_I - U_{I+1}|,
 \]
 where the summation is taken over a batch of samples.
 The loss evaluates the Mean Squared Error (MSE) between the assembled neural flux prediction and the "True" coarse flux computed from fine-mesh simulations. To aggressively target high-frequency interfacial jumps without over-dissipating smooth regions, the per-sample squared error is scaled by a weight factor $w_{I+\oh}$, where $\gamma$ serves as a tunable hyperparameter controlling the sensitivity to local gradients; \revv{the} value $\gamma=0.2$ works well in this case. 
\rev{While offline validation metrics naturally favor $\gamma \to 0$ due to the reduced penalty magnitude, our \textit{a posteriori} analysis demonstrates that under-regularizing the network removes critical physical constraints, causing under-damped high-frequency errors during dynamic integration. Conversely, excessive weighting ($\gamma \ge 0.4$) introduces severe optimization stiffness. Thus, $\gamma = 0.2$ provides the minimum effective regularization necessary for long-term shock stability.}

 Network parameters are optimized via latent-space learning using the Adam optimizer. We implement a cosine decay learning-rate schedule with an initial rate of $10^{-3}$, allowing rapid initial convergence followed by fine-tuning of the network weights. 
 \rev{The dynamic cosine decay schedule decreases the learning rate smoothly from an initial value to a prescribed minimum value according to a cosine curve, providing a gradual transition from aggressive parameter updates to smaller corrective updates during training.}
 The model is trained for 200 epochs using mini-batches of size 128, iteratively minimizing the gradient-weighted loss on the training set. Validation is carried out by performing a long simulation of the reduced equation on the coarse grid and comparing spectra and individual solutions. In the next sections, we use a generic notation $z^{(\cdot)}$ and $\sigma(\cdot)$ to denote hidden layers and activation functions; these design features are different in each section. \rev{A comprehensive summary of the physical parameters, architectural capacity, and training configurations for the proposed model is provided in Table \ref{tab:hyperparameters}.}

 \begin{table}[H]
\rev{\centering
\renewcommand{\arraystretch}{1.2}
\begin{tabular}{llc}
\hline
\textbf{Category} & \textbf{Parameter} & \textbf{Value} \\ \hline
\textbf{Physical Setup} 
& Fine Grid Resolution ($N_f$) & 512 \\
& Coarse Grid Resolution ($N_c$) & 64 \\
& Domain Length ($L$) & $2\pi$ \\
& Time Step ($\Delta t$) & 0.001 \\ \hline
\textbf{Dataset Specifications}
& Training Samples ($80\%$) & 121,832 \\
& Validation Samples ($20\%$) & 30,459 \\ \hline
%
% \textbf{Network Architecture} 
% & Hidden Layer Width  & 16 \\
\textbf{Entropy NN ($\eta$)} 
& Activations & Softplus \\
& 2 Hidden Layers & 16, 16 \\
\textbf{Flux Potential NN ($\phi$)} 
& Activations & Tanh \\
& 1 Hidden Layer & 16 \\
\textbf{Eddy Viscosity NN ($\nu$)} &
Activations & Swish, Sigmoid \\
& 1 Hidden Layer & 16 \\
& Dynamic Viscosity Bounds & $[0.35, 2.0]$ \\ \hline
\textbf{Training \& Optimization}
& Optimizer & Adam \\
& Learning Rate Schedule & Cosine Decay \\
& Initial Learning Rate & $1 \times 10^{-3}$ \\
& Gradient Regularization Weight ($\gamma$) & 0.2 \\
& Batch Size & 128 \\
& Epochs & 200 \\ \hline
\end{tabular} }
\caption{Summary of physical and neural network hyperparameters for the proposed architecture.
\label{tab:hyperparameters}
}
\end{table}

\rev{A sensitivity analysis on hidden neurons ($N \in \{16, 32, 64\}$) for the eddy viscosity NN revealed a flat plateau in validation error. This indicates that the underlying physical manifold of the subgrid flux is fully represented by narrower architectures. To maximize computational acceleration relative to the DNS baseline, we selected the parsimonious 16-neuron architecture, strictly minimizing matrix-multiplication overhead during the PDE time-stepping loop.}  

%%%%%%%%%%%%%%%%%%%
% ICNN
%%%%%%%%%%%%%%%%%%%
\subsubsection{Input Convex Neural Network (ICNN) for Entropy Formulation} 
\label{sec:icnn} 

To ensure that the learned discrete conservation law remains hyperbolic, the parameterized entropy function $\eta_\theta(u)$ must be strictly convex with respect to the state variable $u$. We achieve this by structuring the entropy network as an Input Convex Neural Network (ICNN) \cite{amos2017input}. ICNNs enforce overall convexity by satisfying two structural constraints:
\begin{enumerate}
    \item The nonlinear activation functions must be both convex and non-decreasing.
    \item All weight matrices connecting the hidden layers must be strictly non-negative.
\end{enumerate}

Since we consider a scalar equation, the input is $u \in \bR^1$. We satisfy the first constraint by employing the smooth \texttt{softplus} activation function $\sigma(x) = \ln( 1 + \exp (x))$. In the first hidden layer, the input undergoes a standard affine transformation followed by the activation:
$$ 
z^{(1)} = \sigma\left(W_u^{(1)} u + b^{(1)}\right),
$$
where $W_u^{(1)}$ and $b^{(1)}$ are the standard weights and biases.

For the second hidden layer, the network utilizes a skip-connection, directly re-injecting the input $u$ alongside the features from the previous layer. To satisfy the second convexity constraint, the hidden-to-hidden weight matrix $W_z^{(1)}$ is restricted to the non-negative orthant. In our implementation, this is achieved by taking the element-wise absolute value of the trainable parameters, defining the effective weights as $\tilde{W}_z^{(1)} = |W_z^{(1)}|$. Thus, the second hidden layer evaluates as:
$$ 
z^{(2)} = \sigma\left(z^{(1)} \tilde{W}_z^{(1)} + W_u^{(2)} u + b^{(2)}\right). 
$$
\rev{The dimension of each hidden layer is 16.}

To maximize the expressive power of the ICNN, the final output layer is augmented with both linear and quadratic transformations of the original input $u$. The quadratic weight coefficient, $\tilde{W}_S$, is also constrained to be non-negative via an absolute value projection ($\tilde{W}_S = |W_S|$). The final convex entropy function is computed as:
$$ 
\eta(u) = z^{(2)} \tilde{W}_z^{(2)} + b^{(3)} + W_u^{(3)} u +  \tilde{W}_S u^2. 
$$
Since $\tilde{W}_S \ge 0$ and $u^2$ is a convex function, this quadratic augmentation strictly preserves the overall convexity of the network. This absolute-value bounding strategy provides a robust guarantee of hyperbolicity without the need for complex regularization penalties during training.
\rev{Note that the standard entropy function for the Burgers equation is quadratic, $\eta(u)=\tfrac12 u^2$. Since the goal of our framework is to learn an entropy function appropriate for the coarse discretization, the learned entropy function need not coincide exactly with the standard entropy of the Burgers equation. However, including the quadratic term in the output layer provides the network with the exact Burgers entropy as a baseline, allowing it to learn only the additional corrections required by the coarse model. Therefore, this augmentation, most likely, improves the efficiency of training the entropy network. For other hyperbolic systems, it may be advantageous to augment the output layer with a known convex entropy function of the underlying equations.}

%%%%%%%%%%%%%%%%%%%
% Flux Potential NN
%%%%%%%%%%%%%%%%%%%
\subsubsection{Flux Potential Network} 
\label{sec:fluxnet}

While the entropy network $\eta_\theta(u)$ guarantees the stability of the system, the dispersive wave-speed corrections are governed by the conservative subgrid flux, defined as the gradient of a scalar potential: $\phi'_{\theta}(v)$. 

To parameterize this potential function $\phi_{\theta}$, we use a standard Fully Connected Neural Network (FCNN) mapping the entropy variables $v$ to a scalar output. For this specific implementation, the architecture is kept deliberately lightweight. The input $v$ is passed through a single hidden layer of dimension 16 paired with a hyperbolic tangent $\tanh$ activation function. The hidden state is computed as:
$$ 
z^{(1)} = \tanh\left(W_v^{(1)} v + b^{(1)}\right). 
$$
This is followed by a linear output layer to predict the scalar potential:
$$ 
\phi_{\theta}(v) = W_v^{(2)} z^{(1)} + b^{(2)}.
$$
Taking the exact derivative of this compact network with respect to its inputs provides the conservative flux correction in the reduced equation.

%%%%%%%%%%%%%%%%%%%
% Viscosity NN
%%%%%%%%%%%%%%%%%%%
\subsubsection{Eddy Viscosity Network} 
\label{sec:eddynet}

To suppress spurious oscillations without over-damping the physical energy cascade, we introduce an independent Eddy Viscosity network, $C_{\theta}$. Rather than replacing the baseline numerical dissipation (such as the Local Lax-Friedrichs scheme), this network outputs a dynamic scaling coefficient that modulates the baseline viscosity based on localized feature inputs $\bm{\xi}$ (e.g., cell averages, jumps, or curvature).

The network processes these input features through a single hidden layer of 16 neurons utilizing the \texttt{swish} activation function \cite{ramachandran2017searching}, which ensures smooth, non-vanishing gradient flow during training:
$$ 
z^{(1)} = \text{swish}\left(W_f^{(1)} \bm{\xi} + b_f^{(1)}\right).
$$
A critical design element of this network is the strict physical bounding of its output. To prevent artificial \revv{energy pileup} at the smallest grid scales due to insufficient dissipation, the final layer utilizes a \texttt{sigmoid} activation function $\sigma(x)$ which is explicitly scaled and shifted:
$$ 
C_{\theta} = 1.65 \, \sigma\left(W_f^{(2)} z^{(1)} + b_f^{(2)}\right) + 0.35. 
$$
Since the \texttt{sigmoid} function is strictly bounded by $(0, 1)$, this formulation guarantees that the network output operates exclusively within the range $C_{\theta} \in [0.35, 2.00]$. This architecture ensures that at least $35\%$ of the baseline numerical viscosity is always maintained to suppress spurious high-frequency oscillations.

%%%%%%%%%%%%%%%%%%%
% Analysis
%%%%%%%%%%%%%%%%%%%
\begin{revblock}
\section{Analytical Properties of the Reduced Model}
\label{sec:analyt}
In this section, we describe several analytical properties of the hybrid scheme with flux \eqref{eq:effflux}. In particular, we derive conditions under which the numerical scheme is total variation diminishing and discuss entropy stability.

\revv{The total variation diminishing (TVD) property is a key stability criterion for numerical schemes for hyperbolic conservation laws. A TVD scheme guarantees that the total variation of the numerical solution does not increase in time, thereby suppressing the growth of spurious oscillations near discontinuities.}

\begin{Definition}[Discrete total variation]
Let $\{u_j\}$, $j=0,\ldots,J$ be a discrete sequence. Its total variation is defined by
\[
TV(u)
=
\sum_{j=0}^{J-1}
|u_{j+1}-u_j|.
\]
\end{Definition}

\revv{The following definition introduces the TVD property of a discrete numerical solution $u^n=\{u_0^n,\ldots,u_J^n\}$, where $n$ denotes the time-step index.}

\begin{Definition}[Total variation diminishing scheme]
Let $u_j^n$ denote the numerical approximation at time-step $n$, i.e.,
$u_j^n \approx u(n\Delta t, j\Delta x)$.
A numerical scheme is called total variation diminishing, or TVD, if the total variation of the numerical solution does not increase with time, i.e., 
\[
TV(u^{n+1}) \leq TV(u^n)
\]
for every $n$.
\end{Definition}

First, we need to make an assumption about the regularity of the Neural Subgrid Flux.
\begin{Assumption}[Lipschitz regularity of the Neural Subgrid Flux]
\label{ass:lip}
Consider the Neural Subgrid Flux
$F^{\mathrm{NN}}_{I+\oh} = \calN(U_I, U_{I+1})$, where 
$\calN : \mathbb{R}^2 \to \mathbb{R}$ is a continuous function. 
Assume that for each admissible interval
$\mathcal{B}=[u_{\min},u_{\max}]$, $\calN(a,b)$ is 
Lipschitz in each argument separately; that is,
there is a constant $\Lnn$ (possibly $\Lnn = \Lnn(\mathcal{B})$)
such that
\begin{equation}
\label{Lnn}
\bigl|\calN(a,b) - \calN(c,b)\bigr| \le \Lnn\,|a-c|, \qquad
\bigl|\calN(a,b) - \calN(a,c)\bigr| \le \Lnn\,|b-c|
\end{equation}
for all $a,b,c \in \mathcal{B}$.
\end{Assumption}

Next, we derive TVD conditions for the effective NN scheme and discuss consequences for designing the neural network.
\revv{The main idea is that the neural corrective flux may introduce additional
variation into the numerical solution, while the eddy-viscosity term acts
to damp this variation. The TVD conditions derived below quantify the
amount of numerical dissipation required to compensate for the corrective
flux and guarantee that the total variation does not increase from one
time step to the next.}
%  TVD Monotonicity Proposition -- NN Augmented Scheme

%%%%%%%%%
% NEW VERSION
\begin{Proposition}[Sufficient TVD conditions for the neural network-augmented scheme]
\label{prop:tvd}
Let the spatial domain be discretized with uniform mesh spacing $\Delta x$
and time step $\Delta t$, and define $r=\Delta t/\Delta x$.
Consider the conservative update
\begin{equation}
    U_I^{n+1}=U_I^n -
    r\left(\calF_{I+\oh}-\calF_{I-\oh}\right),
    \label{eq:update}
\end{equation}
where the composite numerical flux at the interface $I+\oh$ is
\begin{equation}
    \calF_{I+\oh} = F_{I+\oh} +
    F^{\mathrm{NN}}_{I+\oh}(U_I,U_{I+1})
    - \calD^{\mathrm{NN}}_{I+\oh}.
    \label{eq:flux_decomp}
\end{equation}
Here
\[
F_{I+\oh} = \frac{1}{2}\bigl(f(U_I)+f(U_{I+1})\bigr),
    \qquad
F^{\mathrm{NN}}_{I+\oh} = \calN(U_I,U_{I+1}),
\]
and the dynamic eddy-viscosity flux is
\begin{equation}
    \calD^{\mathrm{NN}}_{I+\oh}
    =
    C_{\theta,I+\oh}
    \frac{\alpha_{I+\oh}}{2}
    (U_{I+1}-U_I),
    \label{eq:eddy}
\end{equation}
where $C_{\theta,I+\oh}=C_\theta(\boldsymbol{\xi}_{I+\oh})$ is evaluated
from the known numerical solution at time level $n$ and treated as an
interface coefficient during the update. Assume that
\[
    0<C_{\min}\leq C_{\theta,I+\oh}\leq C_{\max}<\infty.
\]
Let
\[
    \Delta U_{I+\oh}=U_{I+1}-U_I
\]
and let 
\(\Lambda_{I+\oh}=\max\{|f'(U_I)|,\ |f'(U_{I+1})|\}\)
be a local wave-speed bound.
To avoid degeneracy of the diffusion \eqref{eq:eddy} when $\Lambda_{I+\oh}=0$, assume that the viscosity
uses a positive speed floor, i.e., 
\[
    \alpha_{I+\oh}=\max(\Lambda_{I+\oh},\varepsilon),
    \qquad \varepsilon>0.
\]
Assume also that the neural corrective flux satisfies Assumption
\ref{ass:lip} on the admissible range, with Lipschitz constant
$\Lnn$.
\revv{The Lipschitz constant $\Lnn$ measures the sensitivity of the
neural corrective flux to perturbations in its inputs. Consequently,
larger values of $\Lnn$ require stronger numerical
dissipation in order to preserve the TVD property.}

If, for every interface $I+\oh$,
\begin{equation}
    C_{\theta,I+\oh} \, \alpha_{I+\oh}
    \geq
    \Lambda_{I+\oh}+2\Lnn,
    \label{eq:cond_C}
\end{equation}
and the time step satisfies
\begin{equation}
    r\left(C_{\theta,I+\oh} \, \alpha_{I+\oh}+2\Lnn\right)
    \leq 1
    \qquad \text{for every interface } I+\oh,
    \label{eq:cond_r}
\end{equation}
then the scheme \eqref{eq:update}--\eqref{eq:flux_decomp} satisfies
Harten's TVD conditions.
\end{Proposition}

\revv{The first condition in \eqref{eq:cond_C} requires that the adaptive eddy viscosity be
large enough to dominate both the physical wave propagation and the
variation introduced by the neural corrective flux. The second condition \eqref{eq:cond_r}
is a CFL-type restriction that ensures the coefficients in Harten's incremental representation satisfy the bounds required for the TVD property.}

\revv{The proof proceeds by rewriting the scheme in Harten's incremental form.
The neural corrective flux is first expressed in terms of neighboring
solution differences using its Lipschitz continuity, after which the
resulting coefficients are compared with Harten's sufficient conditions
for TVD schemes.}

\begin{proof}
\revv{The proof consists of expressing each contribution to the numerical flux
in terms of neighboring solution differences so that the update can be
written in Harten's incremental form. Once this representation is
obtained, the TVD property follows by verifying Harten's coefficient
conditions.}

Define
\[
    \nu_{I+\oh}
    =
    C_{\theta,I+\oh} \, \alpha_{I+\oh}.
\]
Then
\[
    \calD^{\mathrm{NN}}_{I+\oh}
    =
    \frac{\nu_{I+\oh}}{2}\Delta U_{I+\oh}.
\]

\revv{We now express each component of the numerical flux in terms of the neighboring solution differences $\Delta U_{I\pm\oh}$. This allows the update to be written in Harten's incremental form.}

First, for the macroscopic central flux,
\begin{equation}
    F_{I+\oh}-F_{I-\oh}
    =
    \frac{1}{2}a_{I+\oh}\Delta U_{I+\oh}
    +
    \frac{1}{2}a_{I-\oh}\Delta U_{I-\oh},
    \label{eq:macro_diff}
\end{equation}
where $|a_{I+\oh}|\leq \Lambda_{I+\oh}$ from the definition of $\Lambda_{I+\oh}$.

\revv{Next, the neural corrective flux is treated similarly. Since it depends on two neighboring states, we use the Lipschitz regularity assumption to express its difference in terms of the same neighboring increments.}
We write
\[
    F^{\mathrm{NN}}_{I+\oh}-F^{\mathrm{NN}}_{I-\oh}
    =
    \calN(U_I,U_{I+1})-\calN(U_{I-1},U_I).
\]
Adding and subtracting $\calN(U_I,U_I)$ gives
\[
    F^{\mathrm{NN}}_{I+\oh}-F^{\mathrm{NN}}_{I-\oh}
    =
    \bigl[\calN(U_I,U_{I+1})-\calN(U_I,U_I)\bigr]
    +
    \bigl[\calN(U_I,U_I)-\calN(U_{I-1},U_I)\bigr].
\]
Thus there exist coefficients $q_{I+\oh}$ and $p_{I-\oh}$ such that
\begin{equation}
    F^{\mathrm{NN}}_{I+\oh}-F^{\mathrm{NN}}_{I-\oh}
    =
    q_{I+\oh}\Delta U_{I+\oh}
    +
    p_{I-\oh}\Delta U_{I-\oh},
    \label{eq:nn_diff}
\end{equation}
with
\begin{equation}
q_{I+\oh} = \frac{\calN(U_I,U_{I+1}) - \calN(U_I,U_I)}{\Delta U_{I+\oh}}, \qquad
p_{I-\oh} = \frac{\calN(U_I,U_I) - \calN(U_{I-1},U_I)}{\Delta U_{I-\oh}},
\end{equation}
so that
\begin{equation}
    |q_{I+\oh}|\leq \Lnn,
    \qquad
    |p_{I-\oh}|\leq \Lnn
    \label{eq:nn_bounds}
\end{equation}
using Assumption \ref{ass:lip}.
\revv{The coefficients $q_{I+\oh}$ and $p_{I-\oh}$ are difference quotients of $\calN$ with respect to its second and first arguments, respectively.}
When the denominator is zero, the associated numerator is also zero and
the coefficient may be set equal to zero.

The eddy-viscosity contribution satisfies
\begin{equation}
    \calD^{\mathrm{NN}}_{I+\oh}
    -
    \calD^{\mathrm{NN}}_{I-\oh}
    =
    \frac{\nu_{I+\oh}}{2}\Delta U_{I+\oh}
    -
    \frac{\nu_{I-\oh}}{2}\Delta U_{I-\oh}.
    \label{eq:eddy_terms}
\end{equation}

Substituting \eqref{eq:macro_diff}, \eqref{eq:nn_diff}, and
\eqref{eq:eddy_terms} into \eqref{eq:update}, we obtain
\begin{equation}
    U_I^{n+1}
    =
    U_I^n
    +
    \Phi^+_{I+\oh}\Delta U_{I+\oh}
    -
    \Phi^-_{I-\oh}\Delta U_{I-\oh}.
    \label{eq:harten_form}
\end{equation}
\revv{Equation \eqref{eq:harten_form} has the standard incremental form considered by Harten. The coefficients $\Phi^\pm$ determine whether the update is TVD. The coefficients are defined as }
\begin{align}
    \Phi^+_{I+\oh}
    &=
    \frac{r}{2}
    \left(
        \nu_{I+\oh}
        -
        a_{I+\oh}
        -
        2q_{I+\oh}
    \right),
    \label{eq:phi_plus}\\[4pt]
    \Phi^-_{I+\oh}
    &=
    \frac{r}{2}
    \left(
        \nu_{I+\oh}
        +
        a_{I+\oh}
        +
        2p_{I+\oh}
    \right).
    \label{eq:phi_minus}
\end{align}

By Harten's theorem, the incremental scheme
\eqref{eq:harten_form} is TVD provided that, for every interface,
\begin{equation}
    \Phi^+_{I+\oh}\geq 0,
    \qquad
    \Phi^-_{I+\oh}\geq 0,
    \qquad
    \Phi^+_{I+\oh}+\Phi^-_{I+\oh}\leq 1.
    \label{eq:harten_conds}
\end{equation}

\revv{It therefore remains to verify that the coefficients $\Phi^\pm$ satisfy the three inequalities in \eqref{eq:harten_conds}. We estimate each coefficient using the wave-speed bound and the Lipschitz continuity of the neural corrective flux.}
Using
$|a_{I+\oh}|\leq \Lambda_{I+\oh}$ and
$|q_{I+\oh}|\leq \Lnn$, condition \eqref{eq:cond_C} gives
\[
    \nu_{I+\oh}
    -
    a_{I+\oh}
    -
    2q_{I+\oh}
    \geq
    \nu_{I+\oh}
    -
    \Lambda_{I+\oh}
    -
    2\Lnn
    \geq 0.
\]
Therefore $\Phi^+_{I+\oh}\geq 0$. Similarly, using
$|p_{I+\oh}|\leq \Lnn$,
\[
    \nu_{I+\oh}
    +
    a_{I+\oh}
    +
    2p_{I+\oh}
    \geq
    \nu_{I+\oh}
    -
    \Lambda_{I+\oh}
    -
    2\Lnn
    \geq 0,
\]
and hence $\Phi^-_{I+\oh}\geq 0$.

\revv{The final TVD condition concerns the sum of the incremental coefficients.}
Adding \eqref{eq:phi_plus} and \eqref{eq:phi_minus}, the
macroscopic advective terms cancel:
\[
    \Phi^+_{I+\oh}+\Phi^-_{I+\oh}
    =
    r\left(
        \nu_{I+\oh}
        +
        p_{I+\oh}
        -
        q_{I+\oh}
    \right).
\]
Using \eqref{eq:nn_bounds}, we obtain
\[
    \Phi^+_{I+\oh}+\Phi^-_{I+\oh}
    \leq
    r\left(
        \nu_{I+\oh}
        +
        2\Lnn
    \right).
\]
Therefore condition \eqref{eq:cond_r} implies
\[
    \Phi^+_{I+\oh}+\Phi^-_{I+\oh}\leq 1.
\]
All three Harten conditions \eqref{eq:harten_conds} hold, and the
scheme is TVD.
\end{proof}

When $\Lambda_{I+\oh} > \varepsilon$,
condition \eqref{eq:cond_C} can be rewritten as
\[
C_{\theta,I+\oh} \ge 1 + \frac{2\Lnn}{\Lambda_{I+\oh}}.
\]
Inequalities \eqref{eq:cond_C} and \eqref{eq:cond_r} have direct
consequences for the architecture of the network. Recall that we use a rescaled sigmoid function in the final layer of the dissipation network to ensure that 
$C_{\min} \le C_{\theta,I+\oh} \le C_{\max}.$
Thus, the inequality above indicates that it is important to select $C_{\max} > 1$ to make the TVD property feasible.
At the same time, conditions \eqref{eq:cond_C} and \eqref{eq:cond_r} are sufficient, not necessary, and hence they should be interpreted as conservative analytical guidelines. The role of the eddy-viscosity term is primarily to stabilize the scheme near shocks and steep gradients; in smooth regions, excessive artificial diffusion may reduce accuracy. In our numerical experiments, we found that imposing $C_{\min}=1$ makes the scheme too diffusive. Therefore, we choose $C_{\min}<1$ and allow the neural network to increase $C_{\theta,I+\oh}$ adaptively in regions where additional viscosity is needed. The lower bound $C_{\min}$ is also controlled by the next proposition.
Condition \eqref{eq:cond_r} can also be checked a posteriori during the numerical simulation to verify that the time step is sufficiently small.

\noindent We note that strict TVD enforcement is not a necessary
condition for a high-resolution scheme to perform well in practice.
High-order methods such as WENO \cite{shu1998essentially} and
Runge--Kutta Discontinuous Galerkin \cite{cockburn2001runge} are not
strictly TVD \cite{ sweby1984, toro2009}, which is consistent with the theoretical barrier
established by Godunov's theorem \cite{godunov1959}. The conditions
derived in Proposition~\ref{prop:tvd} are sufficient but not necessary:
the scheme may satisfy the TVD criterion locally depending on the
network output $C_{\theta}$, but is not required to do so globally in
order to remain stable and accurate.

\revv{The TVD property controls the growth of oscillations in the numerical
solution, whereas entropy stability guarantees consistency with the
physical entropy inequality satisfied by weak solutions. We next derive
a sufficient condition under which the neural corrective flux and the
adaptive eddy-viscosity term together produce a discretely
entropy-stable scheme.}
\begin{Proposition}[Discrete entropy stability of the neural network-augmented scheme]
\label{prop:entropy}
Consider the semi-discrete conservative scheme for the Burgers equation
\[
u_t+\left(\frac{1}{2}u^2\right)_x=0
\]
with entropy pair
\[
\eta(u)=\frac{1}{2}u^2,
\qquad
q(u)=\frac{1}{3}u^3,
\]
entropy variable $v=\eta'(u)=u$, and entropy potential
\[
\psi(u)=v f(u)-q(u)=\frac{1}{6}u^3.
\]
The semi-discrete scheme is
\begin{equation}
    \frac{dU_I}{dt}
    =
    -\frac{1}{\Delta x}
    \left(
    \calF_{I+\oh}-\calF_{I-\oh}
    \right),
    \label{eq:semi_discrete_entropy}
\end{equation}
where
\begin{equation}
    \calF_{I+\oh}
    =
    F_{I+\oh}
    +
    F^{\mathrm{NN}}_{I+\oh}
    -
    \calD^{\mathrm{NN}}_{I+\oh}.
    \label{eq:flux_ent}
\end{equation}
Here
\begin{equation}
    F_{I+\oh}
    =
    \frac{U_I^2+U_{I+1}^2}{4}
    \label{eq:macro_ent}
\end{equation}
is the central macroscopic flux, and
\begin{equation}
    \calD^{\mathrm{NN}}_{I+\oh}
    =
    C_{\theta,I+\oh}
    \frac{\alpha_{I+\oh}}{2}
    (U_{I+1}-U_I),
    \label{eq:eddy_entropy}
\end{equation}
where
\[
\Lambda_{I+\oh}=\max\{|U_I|,|U_{I+1}|\},
\qquad
\alpha_{I+\oh}=\max\{\Lambda_{I+\oh},\varepsilon\},
\qquad
\varepsilon>0.
\]
Also, assume that the neural corrective flux is constructed using
the convex 
Entropy Neural Network $\eta_\theta$ and the Flux Potential Neural Network $\phi_\theta$ as follows
\begin{equation}
\label{eq:Fnn}
\calN(U_{I+1}, U_I) :=
F^{NN}_{I+\oh} = \frac12 \left[ \phi'_\theta(Z_I) + \phi'_\theta(Z_{I+1}) \right], \quad 
Z_I = \eta'_\theta(U_I).
\end{equation}
Note that here we define the "internal" entropy variables $Z_I$ that, in general, are not the same as the entropy variables for the Burgers equation.

If, for every interface,
\begin{equation}
    C_{\theta,I+\oh} \, \alpha_{I+\oh}
    \geq
    \frac{|U_{I+1}-U_I|}{6} +
    2\Lnn,
    \label{eq:entropy_sufficient}
\end{equation}
then the semi-discrete scheme satisfies a discrete entropy inequality
\begin{equation}
    \frac{d}{dt}\eta(U_I)
    +
    \frac{Q_{I+\oh}-Q_{I-\oh}}{\Delta x}
    \leq 0,
    \label{eq:ent_ineq}
\end{equation}
for a suitable numerical entropy flux $Q_{I+\oh}$.
\end{Proposition}

\revv{The proof follows Tadmor's entropy stability framework. We first identify the entropy-conservative reference flux and then estimate the entropy production associated with each component of the numerical flux. The Lipschitz regularity assumption is used to bound the entropy production contributed by the neural corrective flux. The adaptive eddy viscosity is finally shown to dominate the remaining positive entropy production.}

\begin{proof}
\revv{Following Tadmor's framework, we compare the composite numerical flux
with the entropy-conservative flux for Burgers equation. The entropy
production associated with each component of the numerical flux is then
estimated separately.}

\revv{For the Burgers equation, the physical flux, entropy function, and entropy flux are given by}
\[
f(u)=\frac{1}{2}u^2,
\qquad
\eta(u)=\frac{1}{2}u^2,
\qquad
q(u)=\frac{1}{3}u^3,
\]
respectively.
The entropy variable is $v=u$, and the entropy potential is
\[
\psi(u)=vf(u)-q(u)=\frac{1}{6}u^3.
\]
By Tadmor's condition \cite{tadmor1987numerical}, 
a two-point numerical flux
$F^{\mathrm{EC}}_{I+\oh}$ is entropy-conservative if and only if it
satisfies
\begin{equation}
    \bigl(V_{I+1} - V_I\bigr)\,F^{\mathrm{EC}}_{I+\oh}
        \;=\; \psi_{I+1} - \psi_I.
    \label{eq:tadmor_cond}
\end{equation}
Therefore, using entropy variables and the entropy potential for the Burgers equation, $V_I = U_I$, and
\[
F^{\mathrm{EC}}_{I+\oh}
=
\frac{\psi(U_{I+1})-\psi(U_I)}
     {U_{I+1}-U_I}
=
\frac{U_I^2+U_IU_{I+1}+U_{I+1}^2}{6}.
\]
A scheme satisfies \eqref{eq:ent_ineq} if and only if the total entropy
production at each interface,
\begin{equation}
    \mathcal{E}_{I+\oh}
        \;\coloneqq\;
        (V_{I+1} - V_I)\!\left(\calF_{I+\oh} - F^{\mathrm{EC}}_{I+\oh}\right),
    \label{eq:entropy_prod}
\end{equation}
satisfies $\mathcal{E}_{I+\oh} \le 0$. We evaluate the contribution of
each component of \eqref{eq:flux_ent} to $\mathcal{E}_{I+\oh}$ next.

\textit{\underline{Macroscopic Flux.}}
\revv{We first consider the contribution of the central macroscopic flux.
Since this flux is not entropy conservative, it generates positive
entropy production that must later be compensated by numerical
dissipation.}

The difference between the central macroscopic flux and the
entropy-conservative flux is
\[
F_{I+\oh}-F^{\mathrm{EC}}_{I+\oh}
=
\frac{U_I^2+U_{I+1}^2}{4}
-
\frac{U_I^2+U_IU_{I+1}+U_{I+1}^2}{6}
=
\frac{(U_{I+1}-U_I)^2}{12}.
\]
Thus the entropy production associated with the central macroscopic flux is
\begin{equation}
    (U_{I+1}-U_I)
    \left(
    F_{I+\oh}-F^{\mathrm{EC}}_{I+\oh}
    \right)
    =
    \frac{(U_{I+1}-U_I)^3}{12}.
    \label{eq:central_entropy_prod}
\end{equation}

%%%%%%%%%
% Neural Entropy Flux

\textit{\underline{Neural Corrective Flux.}}
\revv{Next, we estimate the entropy contribution of the neural corrective
flux. Unlike the macroscopic flux, this contribution depends on the
learned neural approximation and is controlled using the Lipschitz
regularity assumption.}

From \eqref{eq:Fnn}, the Neural Corrective Flux
satisfies the neural entropy production equation
\begin{equation}
\label{eq:R}
    \calR_{I+\oh}:=(V_{I+1} - V_I)\,F^{\mathrm{NN}}_{I+\oh}
        = (U_{I+1} - U_I)
        \frac{\phi'_{\theta}(Z_I) + \phi'_{\theta}(Z_{I+1})}{2}.
\end{equation}
Next, using the definition of the neural flux, we can write 
$\phi'_{\theta}(Z(U)) = \calN(U,U)$,
and using Assumption \ref{ass:lip} we can bound the difference 
$|\phi'_{\theta}(Z(U_{I+1})) - \phi'_{\theta}(Z(U_{I}))| = |\calN(U_{I+1},U_{I+1}) - \calN(U_{I},U_{I})| \le 2\Lnn|U_{I+1} - U_{I}|$.
Therefore, 
\begin{equation}
|\calR_{I+\oh}| \le \Lnn(U_{I+1} - U_I)^2.
\end{equation}

\textit{\underline{Dissipative Flux and Entropy Production.}}
\revv{Combining the previous estimates, the total entropy production of the
composite numerical flux is bounded by}
\[
\calR_{I + \oh} +
\frac{(U_{I+1}-U_I)^3}{12}
-
(U_{I+1}-U_I)\calD^{\mathrm{NN}}_{I+\oh}.
\]
Treating the worst-case scenario when the neural entropy production is $\calR_{I+\oh} > 0$, 
the Tadmor entropy stability condition is therefore
\begin{equation}
    (U_{I+1}-U_I)\calD^{\mathrm{NN}}_{I+\oh}
    \geq
    \frac{(U_{I+1}-U_I)^3}{12} + |\calR_{I+\oh}|.
    \label{eq:diss_bound}
\end{equation}
Using \eqref{eq:eddy_entropy}, we obtain
\(
    (U_{I+1}-U_I)\calD^{\mathrm{NN}}_{I+\oh}
    =
    \tfrac12 C_{\theta,I+\oh}
    \alpha_{I+\oh}
    (U_{I+1}-U_I)^2.
\)
\revv{The eddy-viscosity term therefore compensates both the entropy
production of the central flux and the worst-case positive contribution
of the neural corrective flux.}
Thus a sufficient pointwise condition for \eqref{eq:diss_bound} is
\[
    C_{\theta,I+\oh}
    \frac{\alpha_{I+\oh}}{2}
    (U_{I+1}-U_I)^2
    \geq
    \frac{|U_{I+1}-U_I|^3}{12} + \Lnn 
    (U_{I+1} - U_I)^2,
\]
which is equivalent, for $U_{I+1}\neq U_I$, to
\begin{equation}
\label{eq:entropy_sufficient2}  C_{\theta,I+\oh} \, \alpha_{I+\oh}
    \geq
    \frac{|U_{I+1}-U_I|}{6} + 2\Lnn.
\end{equation}
If $U_{I+1}=U_I$, the inequality is trivially satisfied using the definition of $\calR_{I+\oh}$ in \eqref{eq:R}.

\revv{Hence the adaptive dissipation dominates all possible positive entropy
production, implying the discrete entropy inequality
\eqref{eq:ent_ineq}.}

\end{proof}
%%%%%%%%%
% end proof

\revv{Condition \eqref{eq:entropy_sufficient} requires the adaptive eddy
viscosity to provide enough dissipation to offset both the entropy
generated by the central discretization and the additional entropy
production that may arise from the neural corrective flux.}

Consider the lower bound with $\calR_{I+\oh} \le 0$. Then, 
for $\Lambda_{I+\oh} \ge \varepsilon$, 
the condition \eqref{eq:entropy_sufficient} is equivalent to 
\[
C_{\theta,I+\oh} \ge \frac{|U_{I+1} - U_I|}{6\,\Lambda_{I+\oh}}
\]
with $\Lambda_{I+\oh} = \max(|U_{I+1}|, |U_I|).$
The triangle inequality gives
$|U_{I+1} - U_I| \le |U_I| + |U_{I+1}| \le 2\Lambda_{I+\oh}$. Thus, the standard sufficient condition for discrete entropy stability of the central flux approximation becomes 
\[
C_{\theta,I+\oh} \ge \frac{2\Lambda_{I+\oh}}{6\Lambda_{I+\oh}} = \frac13.
\]
Thus, condition 
\eqref{eq:entropy_sufficient} 
(same as eq. \eqref{eq:entropy_sufficient2}) is essentially equivalent to 
\[
C_{\theta,I+\oh} \ge \frac13 + \frac{2\Lnn}{\Lambda_{I+\oh}}.
\]
Therefore, we select $C_{\min} = 0.35$ as the lower bound for the eddy viscosity network. Numerically, as shown in the next section, this bound is sufficient to ensure that no oscillations develop near shocks.
\end{revblock}

%%%%%%%%%%%%%%%%%%%
% Numerical results
%%%%%%%%%%%%%%%%%%%
\section{Numerical Results}
\label{sec:num}
In this section, the performance of the neural network-based reduced model is evaluated by comparing its stationary statistical properties with those of the full model. We perform long stationary simulations for the full and reduced models and perform a detailed comparison of several statistical properties.  The time step for the simulations of the reduced models is consistent with the time step used in the numerical integration of the full model. 

For our numerical experiments with the Burgers' equation, we used the fine-mesh discretization $ N_f = 512 $ in simulations of the full model. This high-resolution grid ensures that all relevant physical features are captured. The numerical diffusion in this fine-mesh setup is negligible, meaning further refinement would yield minimal improvements. We generate coarse-grid representations with \( q = 8 \). Simulations on the fine-mesh grid are referred to as DNS (Direct Numerical Simulation), serving as a reference for comparison. Reduced model simulations are denoted as {LLF-$64$} for the Local Lax-Friedrichs method with 64 points, and as {NN-$64$} when the neural network subgrid model is used. Since the neural network framework is structure-preserving, there is no risk that the NN-reduced model may generate non-physical behaviors.

\textbf{Energy Spectra.}
To quantify the model's statistical accuracy, we evaluate the distribution of kinetic energy across different spatial scales. The energy spectrum is computed in Fourier space by applying the discrete Fourier transform (DFT) to the state variable $u(x,t)$ at each recorded time step, yielding the Fourier coefficients $\hat{u}_k(t)$ for each wavenumber $k$. 
The time-averaged spectral energy 
is calculated as:
\begin{equation*}
    E_k = \frac{1}{T} \int_{0}^{T} |\hat{u}_k(t)|^2 \, dt \approx \frac{1}{N_T} \sum_{i=1}^{N_T} |\hat{u}_k(t_i)|^2
\end{equation*}

We generate a long, stationary time series to ensure statistical convergence. The system is integrated for $T = 1000$ using a time step of $\Delta t = 0.001$. To compute the time-averaged spectra, we sample data with $\Delta t^{\text{sample}} = 0.1$, yielding $N_T = 10,000$ discrete snapshots. 
The total macroscopic energy of the system is the sum of the spectral energy over all 
wavenumbers: $ E_{\text{total}} = \sum_k E_k$. 

Figure \ref{fig:energy_spectra_base} compares the time-averaged energy spectra of the high-resolution DNS, the low-resolution LLF-64 simulation, \rev{the TVD-64 scheme by van Leer \cite{van1979towards},} and the simulation of the reduced NN-64 model. 
The LLF-64 scheme is clearly too diffusive, artificially draining energy from wavenumbers starting with $|k| \ge 6$ and causing a premature spectral decay. 
\rev{The TVD-64 scheme is less diffusive than the LLF-64 simulation; nevertheless, it still underestimates the spectra and the total energy at the resolution $N_c = 64$.}
This clearly demonstrates the need for subgrid modeling if the reduced resolution $N=64$ is the target. 
The NN-64 reduced model accurately resolves the inertial range, tracking the DNS energy cascade precisely up to the grid cut-off without exhibiting spurious energy accumulation.
We present the relative error for the total energy in each system in Table \ref{tab:total_energy}.
The NN-64 subgrid modeling framework yields a low relative error compared to the LLF-64 model. This confirms that the localized Eddy Viscosity network correctly balances energy dissipation, ensuring physical conservation principles are maintained at the macroscopic level.

% ====================
% FIGURE Spectra
% ====================
\begin{figure}[H]
    \centering
     \includegraphics[width=0.7\textwidth]{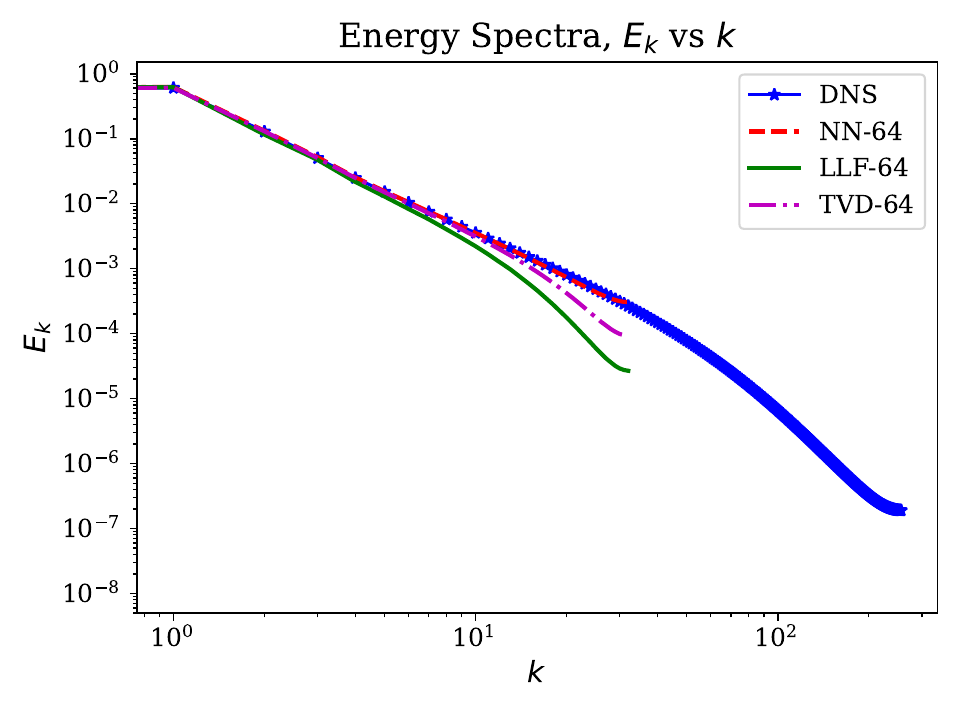} 
    \caption{Energy spectra $E_K$ in simulations of the fully resolved model with $N_f=512$, reduced NN-64 model, \rev{TVD-64 scheme,} and coarse LLF-64 simulation. The NN-64 parameterization effectively corrects the excessive numerical dissipation inherent to the coarse LLF-64 \rev{and TVD-64} schemes, perfectly recovering the energy cascade.}
    \label{fig:energy_spectra_base}
\end{figure}

% ====================
% TABLE Energy
% ====================
\begin{table}[H]
    \centering
    \begin{tabular}{lcc}
        \hline
        \textbf{Model} & \textbf{Total Energy ($E_{\text{total}}$)} & Relative Error with DNS\\
        \hline
        DNS & 0.8830 & \\
        NN-64 Model  & 0.8746 & 0.95\%\\
        LLF-64       & 0.8466 & 4.12 \% \\
        TVD-64       & 0.8696 & 1.5 \% \\
        \hline
    \end{tabular}
    \caption{Comparison of total energy and the relative error for the different numerical models.}
    \label{tab:total_energy}
\end{table}

\rev{The TVD-64 \revv{scheme} represents the Monotonic Upstream-Centered Scheme for Conservation Laws (MUSCL), originally formulated by van Leer \cite{van1979towards}. The MUSCL scheme extends the first-order Godunov approach to achieve second-order spatial accuracy via a piecewise linear reconstruction of the state variables at the cell interfaces. In our implementation, we employ the classical van Leer limiter \cite{van1974towards}, which dynamically enforces the Total Variation Diminishing (TVD) condition by reverting to first-order accuracy in regions of steep gradients or local extrema. The reconstructed interface states are subsequently evaluated using the Local Lax-Friedrichs (LLF) solver to compute the numerical fluxes, ensuring a robust and entropy-stable baseline for performance comparison.}

\begin{revblock}
\subsection{Comparison with Other Subgrid-Scale Models}

To explicitly demonstrate the limitations of traditional closures, we compare the proposed neural architecture against classical static and dynamic Smagorinsky models \cite{germano1991dynamic, lilly1992proposed,smagorinsky1963}. Both models are applied as explicit corrective fluxes on top of the Local Lax-Friedrichs (LLF) \cite{rusanov1961calculation} baseline scheme on the coarse $N=64$ grid.

In addition, Wasserstein Generative Adversarial Network (WGAN) \cite{alcala2021subgrid} and 
Stochastic Mode Reduction (SMR) \cite{dat12} have been used to model subgrid processes for the same coarse variables \eqref{eq:U} in the finite-volume discretization of the forced 1D Burgers equation. Overall, the WGAN and SMR parametrizations perform well in reproducing the spectra and temporal correlations. However, training the WGAN deep learning model was particularly challenging. The SMR approach also performed well, but it
slightly overestimated the energy of higher wavenumbers in the spectra.
Moreover, we found that when coarse variables are defined as local spatial averages, it may be challenging to compute SMR interaction coefficients for systems of equations and multi-dimensional problems since the SMR approach is a semi-analytical technique that requires knowledge of the statistics of fluctuations
$y_i = u_i - U_{I(i)}$.
Finally, an analytical investigation of both the WGAN and SMR parametrizations is extremely challenging. 

\textbf{Static Smagorinsky.} The classical static Smagorinsky model
\cite{smagorinsky1963}
computes artificial viscosity directly proportional to the absolute local velocity gradient, given by the formulation $\nu_{sgs} = (C_s \Delta x)^2 \left| \frac{\partial u}{\partial x} \right|$. We first evaluate the static Smagorinsky model across a range of fixed coefficients. \revv{Figure \ref{fig:spectra_comparison_static} 
presents the energy spectra for $C_s \in \{0.01, 0.05, 0.15, 0.20\}$.} For the inviscid Burgers' equation, this spatial gradient is naturally maximized at the shock front.
Consequently, the static model applies its heaviest numerical damping exactly where the wave is steepest. For artificially low values ($C_s=0.01$), the subgrid contribution is negligible, leaving the overly dissipative LLF baseline unchanged. However, as $C_s$ approaches standard theoretical values ($C_s=0.15$ and $0.20$), the excess viscosity severely degrades the wave amplitude.
\begin{figure}[H]
    % Top Row: 0.01 and 0.05
    \centerline{
    \includegraphics[width=0.5\linewidth]{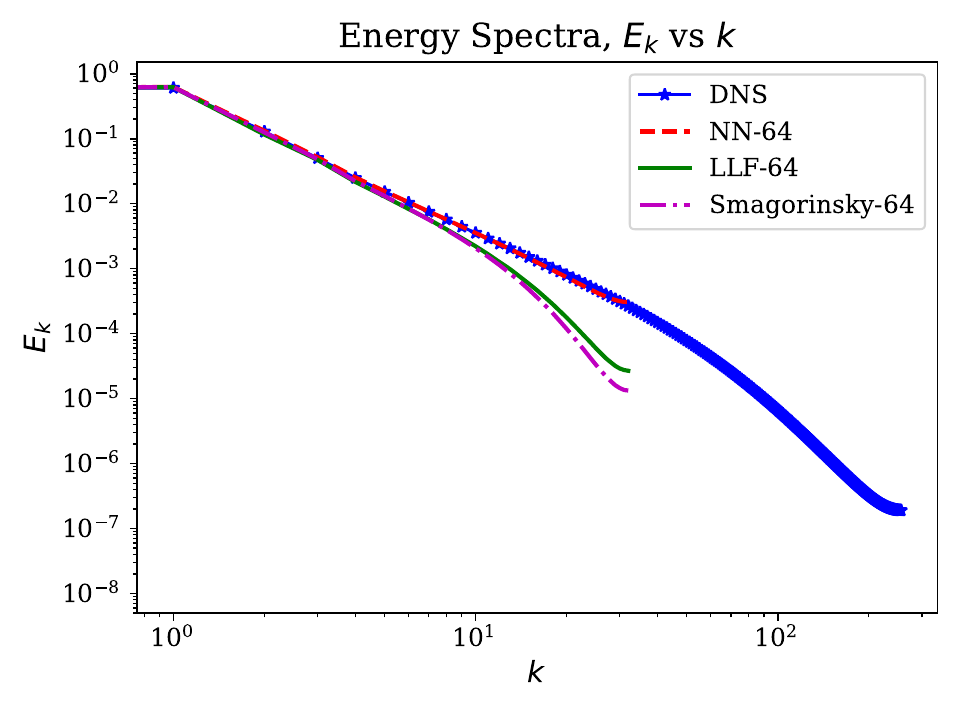} 
    \includegraphics[width=0.5\linewidth]{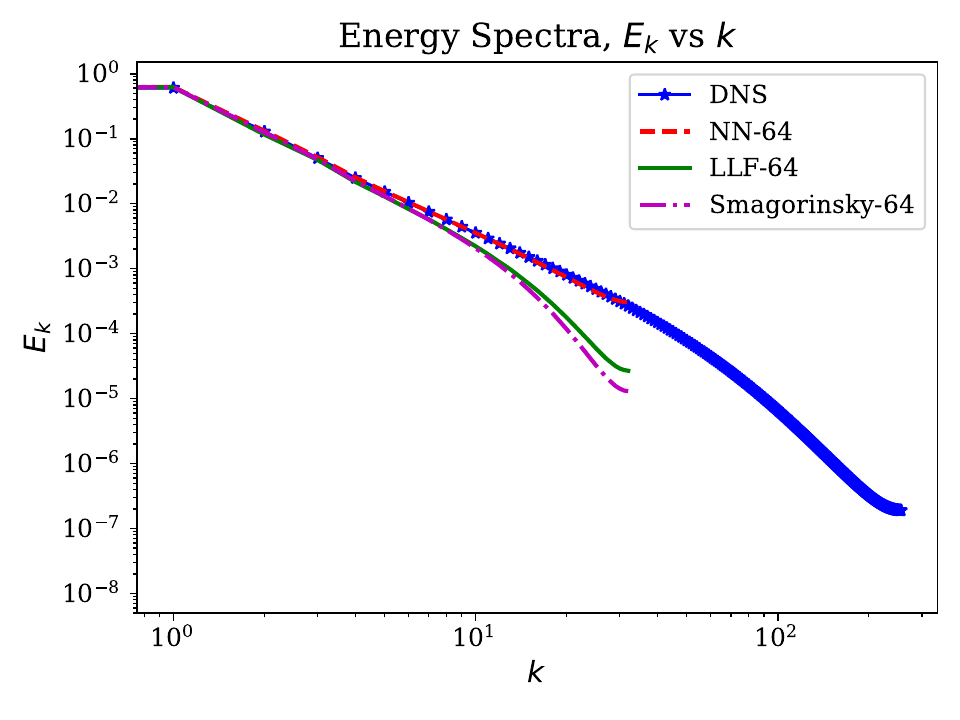}}
    
    % Bottom Row: 0.15 and 0.20
    \centerline{
    \includegraphics[width=0.5\linewidth]{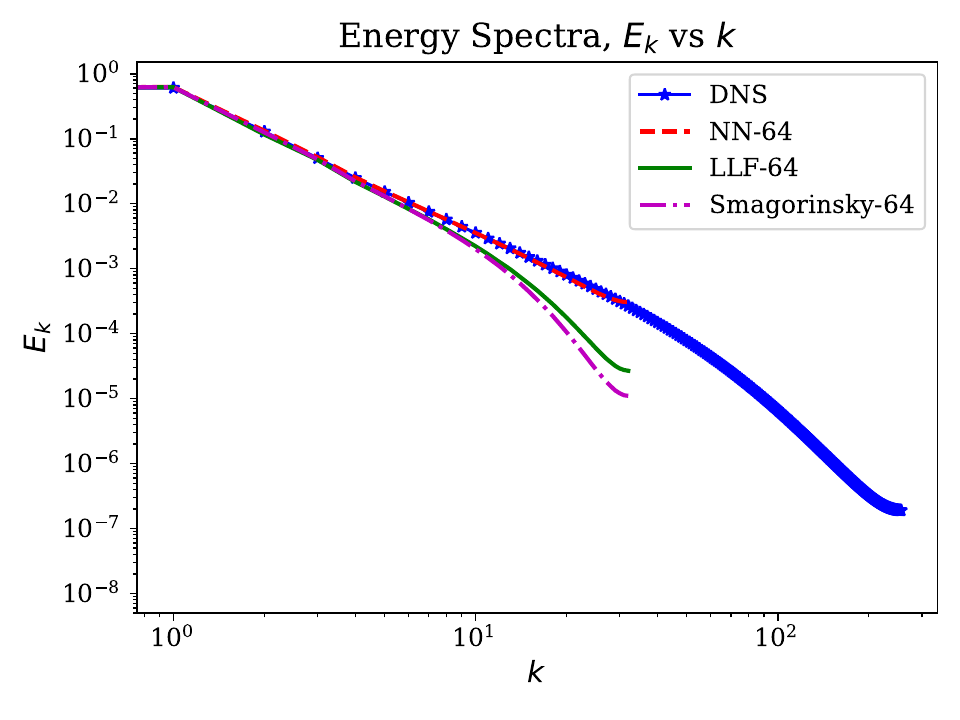} 
    \includegraphics[width=0.5\linewidth]{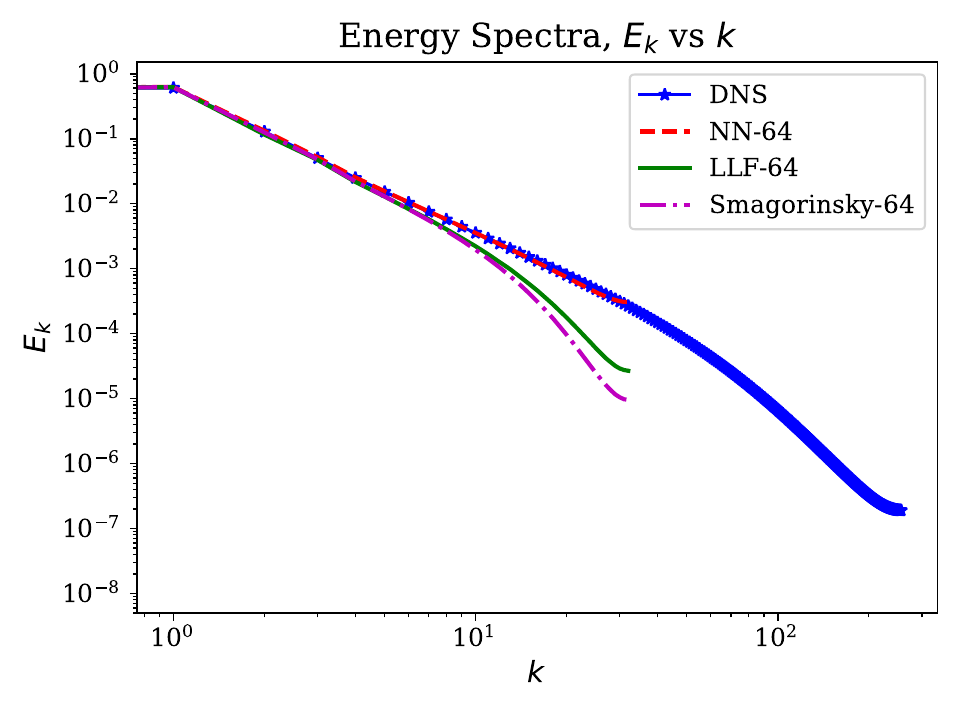}}
    
   \caption{Energy spectra $E_K$ in simulations of the fully resolved model with $N_f=512$ (blue), reduced NN-64 model (red; nearly overlaps with blue), coarse LLF-64 simulation (green), \rev{and Static Smagorinsky model (magenta) with $C_s=0.01$ (top left),
   $C_s=0.05$ (top right),  $C_s=0.15$ (bottom left), and
   $C_s=0.2$ (bottom right).}
The NN parameterization effectively corrects the excessive numerical dissipation inherent to the coarse LLF-64 scheme, perfectly recovering the energy cascade.}
    \label{fig:spectra_comparison_static}
\end{figure}

\textbf{Dynamic Smagorinsky Model.}
We next evaluate the Dynamic Smagorinsky Model (DSM) \cite{germano1991dynamic, lilly1992proposed}. Unlike the static variant, DSM dynamically computes the coefficient $C_s(x,t)$ utilizing a dual-filtering procedure. By defining a grid-filter scale $\Delta$ and a test-filter scale $\widehat{\Delta}$ (where typically $\widehat{\Delta} = 2\Delta$), the resolved turbulent stresses at the test-filter level, $T$, and the grid-filter level, $\tau$, are related exactly by Germano's identity: $\mathcal{L} = T - \widehat{\tau} = \widehat{\overline{u}\overline{u}} - \widehat{\overline{u}}\widehat{\overline{u}},$
where $\mathcal{L}$ denotes the resolved Leonard stress. 

Applying the Smagorinsky closure at the grid-filter and test-filter levels gives
\[
\mathcal{L} \approx C_s^2 M, \quad
\text{where} \quad
M = -2 \widehat{\Delta}^2
\left|\partial_x \widehat{\overline{u}}\right|
\partial_x \widehat{\overline{u}}
+ 2 \Delta^2
\widehat{
\left|\partial_x \overline{u}\right|
\partial_x \overline{u}
} .
\]
Following the least-squares argument of Lilly \cite{lilly1992proposed}, the dynamic coefficient is obtained by minimizing the Germano-identity residual,
\(
\left(\mathcal{L}-C_s^2 M\right)^2 .
\)
In the present one-dimensional setting this yields the local scalar estimate
\[
C_s^2(x,t)
=
\frac{\mathcal{L}(x,t)M(x,t)}{M^2(x,t)},
\]
whenever \(M\neq 0\). Equivalently, in strictly scalar form, this reduces to \(C_s^2=\mathcal{L}/M\). In practice, one may introduce spatial, temporal, or ensemble averaging in the numerator and denominator to regularize this estimate.

\revv{Figure \ref{fig:dynamic_smag}} illustrates the failure of this local DSM formulation for the 1D Burgers test case considered here. The dynamic procedure assumes scale similarity of the modeled SGS stress between the grid-filter and test-filter levels. This assumption is reasonable in inertial-range turbulence but becomes questionable near under-resolved Burgers shocks, where the filtered gradient is dominated by the shock thickness rather than by a self-similar cascade. Across such sharp interfaces, the Leonard stress \(\mathcal{L}=\widehat{\bar u^2}-\widehat{\bar u}^{,2}\) measures the finite resolved variance introduced by the test filter, while the modeled tensor $M$ is constructed from nonlinear functions of the filtered gradient. The resulting local ratio \(C_s^2=\mathcal{L}M/M^2\) can therefore become poorly conditioned or change sign near shocks.

When \(\mathcal{L}M<0\), the dynamic procedure yields a negative eddy viscosity, corresponding to local backscatter or anti-diffusion. In a shock-dominated Burgers solution, such anti-diffusive contributions are numerically dangerous because they can oppose the entropy-dissipative regularization needed to obtain the physically admissible solution. Applying ad-hoc clipping removes these negative-viscosity regions, but it does not control large positive values of $C_s^2$. Consequently, localized peaks in the dynamic coefficient near the shock can produce excessive artificial viscosity, leading to over-damping and distortion of the resolved wave.
%
% --- DYNAMIC SMAGORINSKY ---
\begin{figure}[H]
    
    \centering
    \includegraphics[width=0.75\textwidth]{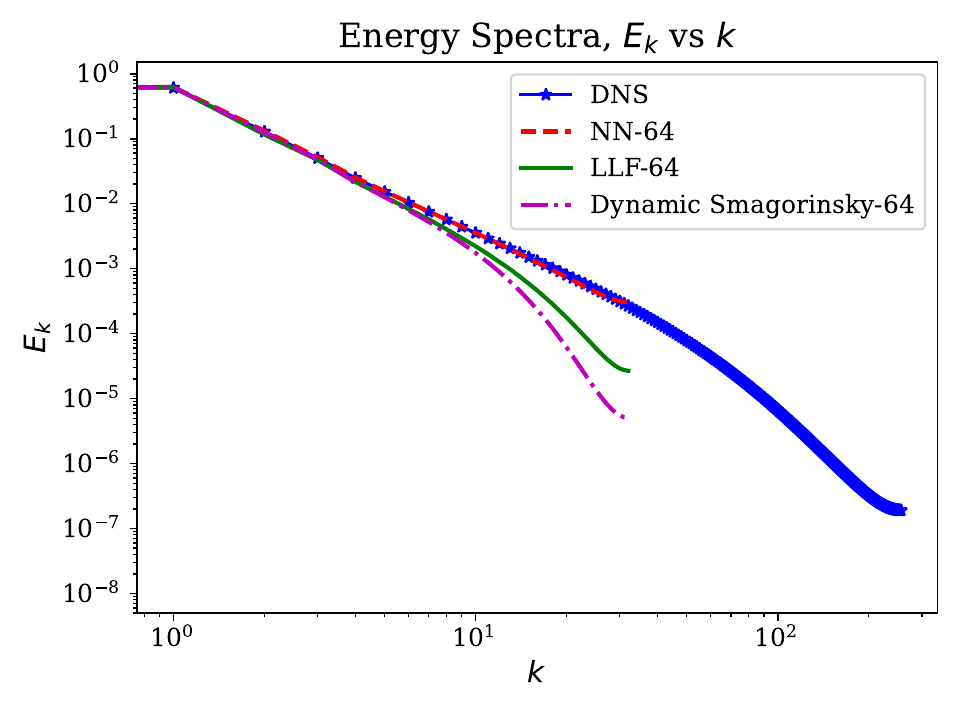}
    \caption{Energy spectra $E_K$ in simulations of the fully resolved model with $N_f=512$ (blue), reduced NN-64 model (red; nearly overlaps with blue), coarse LLF-64 simulation (green), and Dynamic Smagorinsky model (magenta). The failure of the Dynamic Smagorinsky Model (DSM) is likely caused by applying the test filter across sharp discontinuities, where the scale-similarity assumption breaks down; the resulting Leonard stress is dominated by the finite shock jump, leading to poorly conditioned dynamic-coefficient estimates and localized excess viscosity that distorts the resolved wave structure.}
\label{fig:dynamic_smag}
\end{figure}
\end{revblock}

\textbf{Spatial Correlation Function.}
To validate how well spatial structures are captured by the reduced model, we analyze the two-point spatial correlation function. This statistical metric quantifies the dependence and spatial coherence between different locations across the domain.

For the scalar state variable $U(x,t)$, the time-averaged spatial correlation evaluated at a discrete grid point $X_I$ is defined as:
\begin{equation}
\label{cfspace}
    C(X_I) = \frac{1}{N_T} \sum_{k=1}^{N_T} U(X_0, t_k) U(X_I, t_k),
\end{equation}
where $N_T$ represents the total number of sampled time snapshots. 
Because the stochastically forced system is spatially homogeneous and operates within a periodic domain, its statistical properties are translationally invariant. Therefore, computing the correlation with the first spatial grid point $X_0$ is sufficient to characterize the entire domain. 

Figure \ref{fig:corr} (left) presents the normalized spatial correlation function for the resolved flow field. The curve exhibits the characteristic symmetric profile expected in a periodic domain, reflecting the spatial decorrelation length and the recovery of coherency across the periodic boundary.
\rev{Both 
the reduced NN-64 and coarse LLF-64 models show} agreement with the high-resolution DNS reference. \rev{This implies that the NN corrections are not affecting the spatial correlation significantly, and} the reduced model accurately reproduces the spatial correlation scales. This confirms that our structure-preserving architecture successfully captures the macroscopic spatial scales and physical coherence of the flow field, despite operating on a truncated grid resolution.

% ====================
% FIGURE Spatial and Temporal Correlation
% ====================
\begin{figure}[htbp]
    \centerline{
     \includegraphics[width=0.5\textwidth]{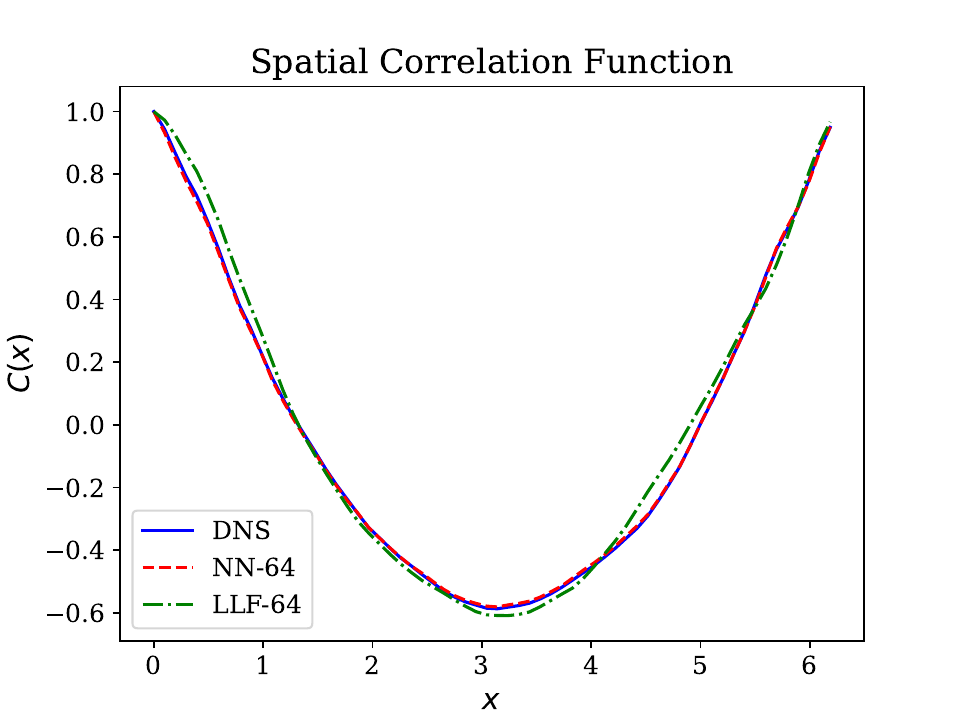} 
     \includegraphics[width=0.5\textwidth]{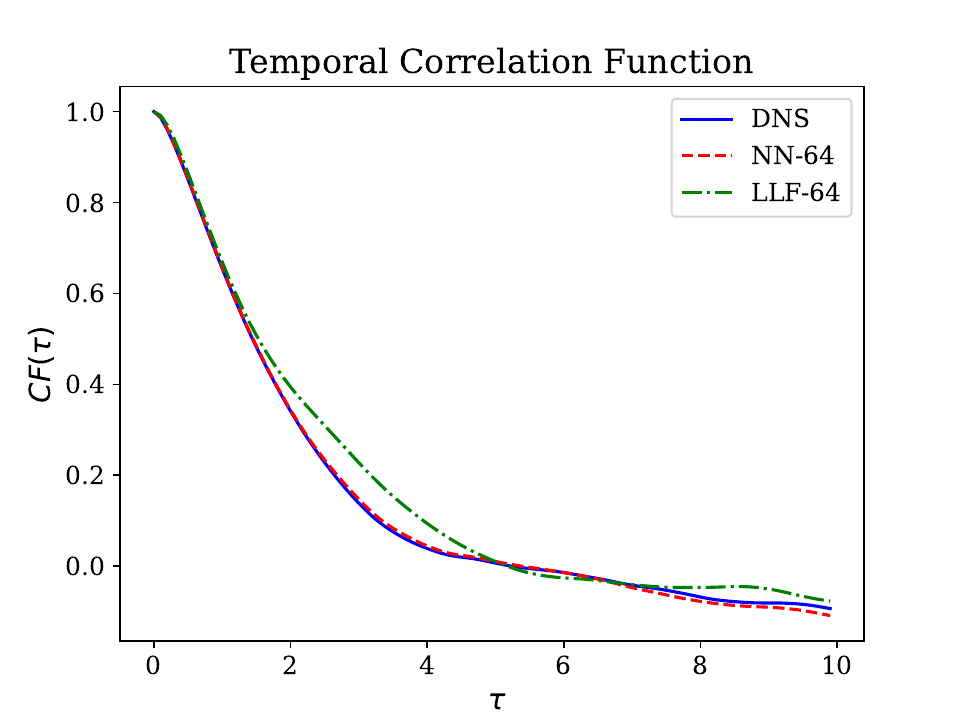} }
    \caption{Left: Spatial Correlation Function for coarse variables $U_I$ given by \eqref{cfspace}, 
    Right: Temporal Correlation Function of coarse variables $U_0$ given by \eqref{cftime}.
    Solid Blue line: high-resolution DNS simulations,
    \rev{Dashed Green line: simulations of the LLF-64 model,}
    Dashed Red line: simulations of the NN-64 reduced model.}
    \label{fig:corr}
\end{figure} 

\textbf{Temporal Correlation Function.}
To evaluate the dynamical consistency of the reduced model, we compute the temporal correlation function. This statistical metric quantifies the system's memory. It indicates how rapidly dynamical variables decorrelate from their prior states. Since the Burgers' equation \revv{is} spatially homogeneous, we consider only the 
spatial location $U(X_0, t)$.
The time correlation function is defined as:
\begin{equation}
\label{cftime}
    CF(\tau) = \frac{1}{T} \int_{0}^{T} U(X_0,t) U(X_0,t+\tau) \, dt .
\end{equation}
Computationally, the time-correlation function is computed using a moving-window approach. 

The normalized temporal correlation functions for the coarse variables in NN-64, LLF-64, and DNS are presented in the right part of Figure \ref{fig:corr}. 
The NN-64 reduced model demonstrates very good agreement with the high-resolution DNS simulation. \rev{For the temporal correlation function, a slight discrepancy is observed between the LLF-64 and DNS curves, indicating that the NN correction terms improve the temporal correlation statistics.}
The NN parameterization accurately captures the decorrelation timescale and the decay profile of the correlation function. Numerical results with spatial and temporal correlation functions confirm that the reduced-order framework preserves not only the spatial energy cascade but also the accurate statistical evolution and temporal dynamics of the full-scale system.

\textbf{Solution Snapshots.}
Neural network approximations are often treated as data-driven "black-box" models without 
incorporating physical constraints.
As a consequence, neural network approximations, when left uncontrolled, can produce unbounded oscillations or unrealistic solutions. 
One possible approach is to employ flux limiters \cite{mojamder2026subgrid,timofeyev2026subgrid} to 
suppress non-physical behavior and spurious oscillations. 
In this paper, we demonstrate that our approach does not require flux-limiting stabilization techniques. This is a direct consequence of using structure-preserving neural networks, which \revv{enforce} stability at the architectural level.
The structure-preserving neural network design used in this paper guarantees that the subgrid parameterization generates robust, physically consistent solutions, replacing the need for traditional flux-limiting procedures with rigorous architectural constraints.

Figure \ref{fig:burgers_snapshots}
depicts typical solution profiles for four selected times. 
\rev{For comparison, we also include the LLF-64 coarse-model results, which bear little resemblance to the DNS solution and fail to reproduce the relevant dynamics.}
We can see that 
the NN model demonstrates excellent agreement with the DNS, successfully capturing shock formations and nonlinear wave propagation. The embedded Eddy Viscosity network is essential in naturally suppressing spurious oscillations near steep gradients, eliminating the need for external flux limiters.

% ====================
% FIGURE Solution Snapshots
% ====================
\begin{figure}[H]
    % First row
    \centerline{
    \includegraphics[width=0.5\textwidth]{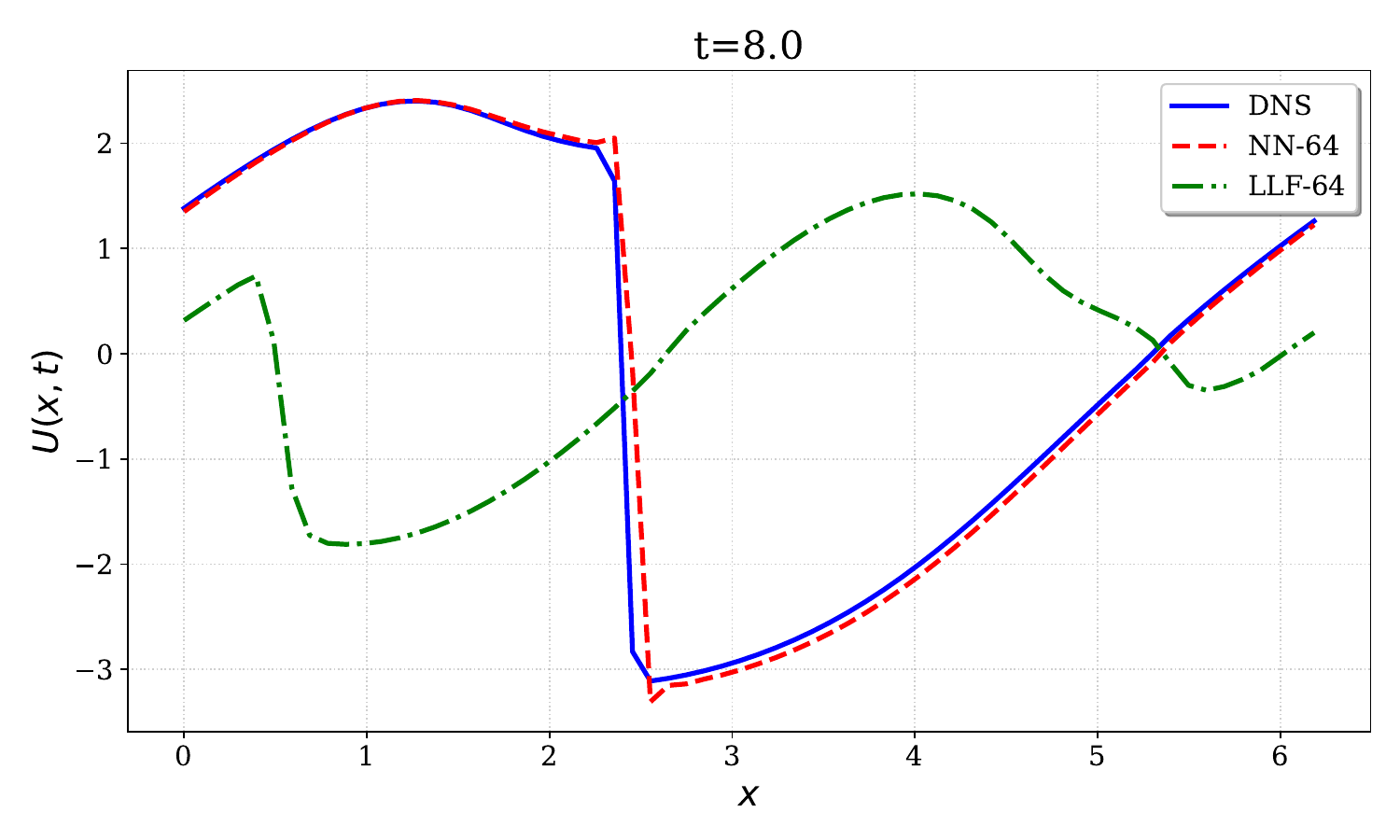}
    \includegraphics[width=0.5\textwidth]{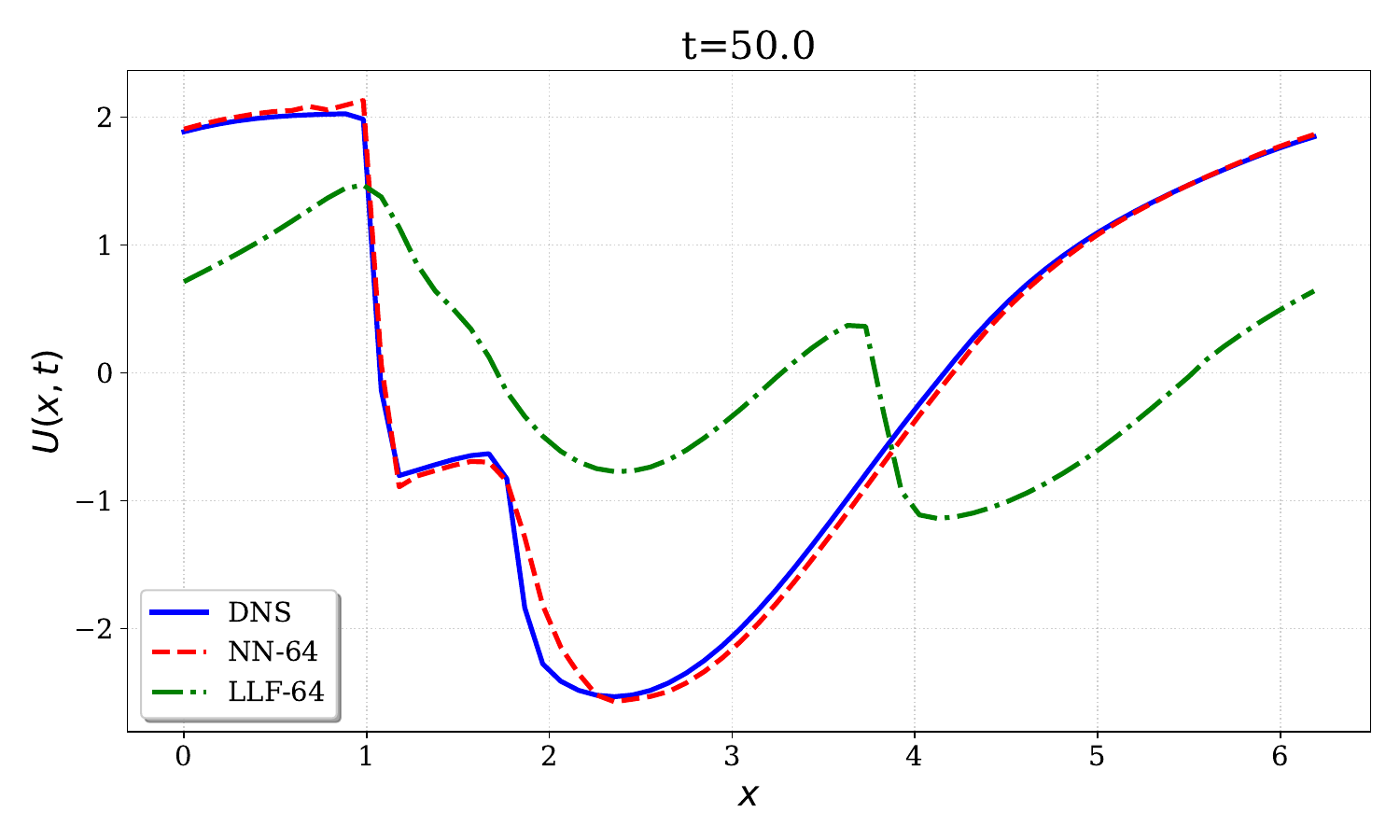}}   
    % Second row
    \centerline{
    \includegraphics[width=0.5\textwidth]{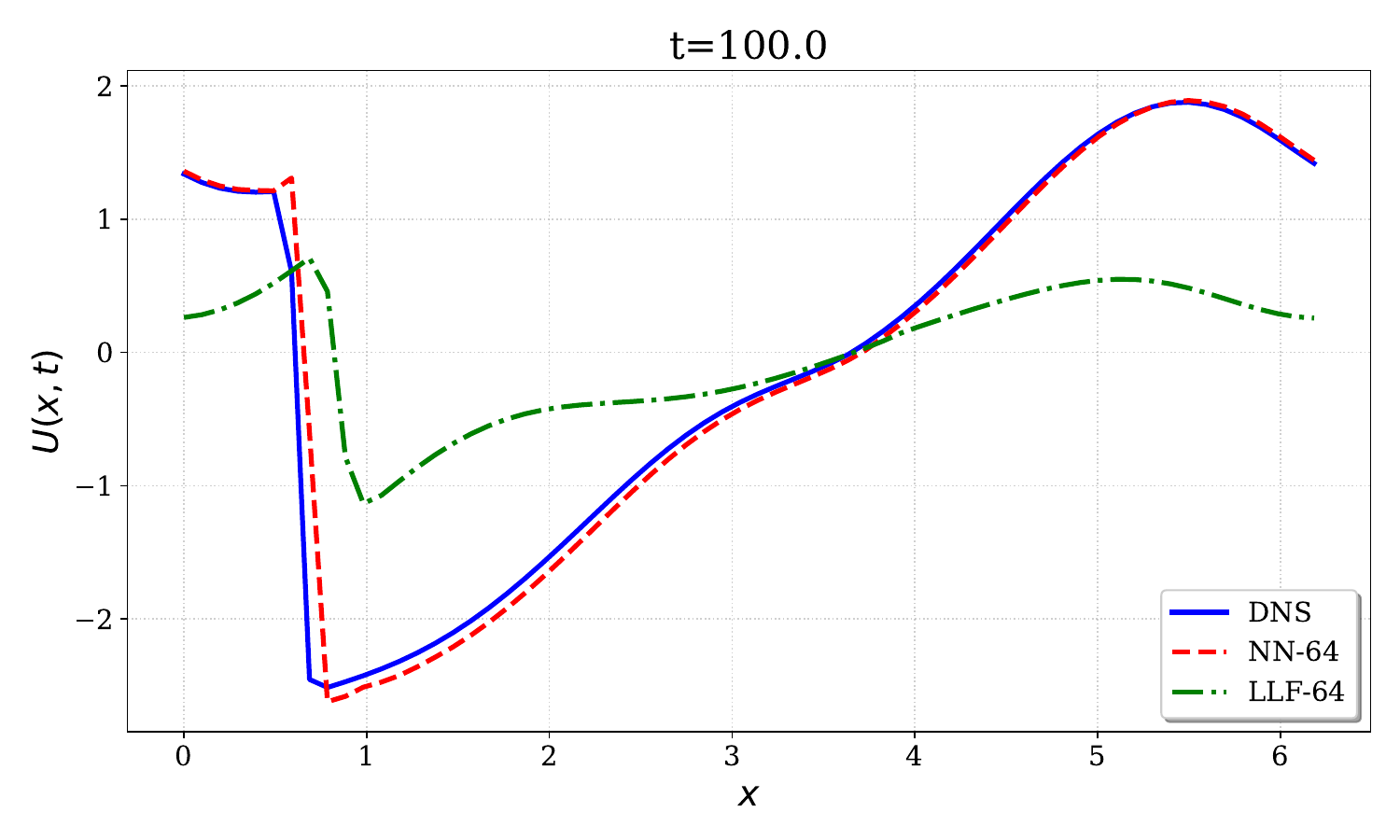}
    \includegraphics[width=0.5\textwidth]{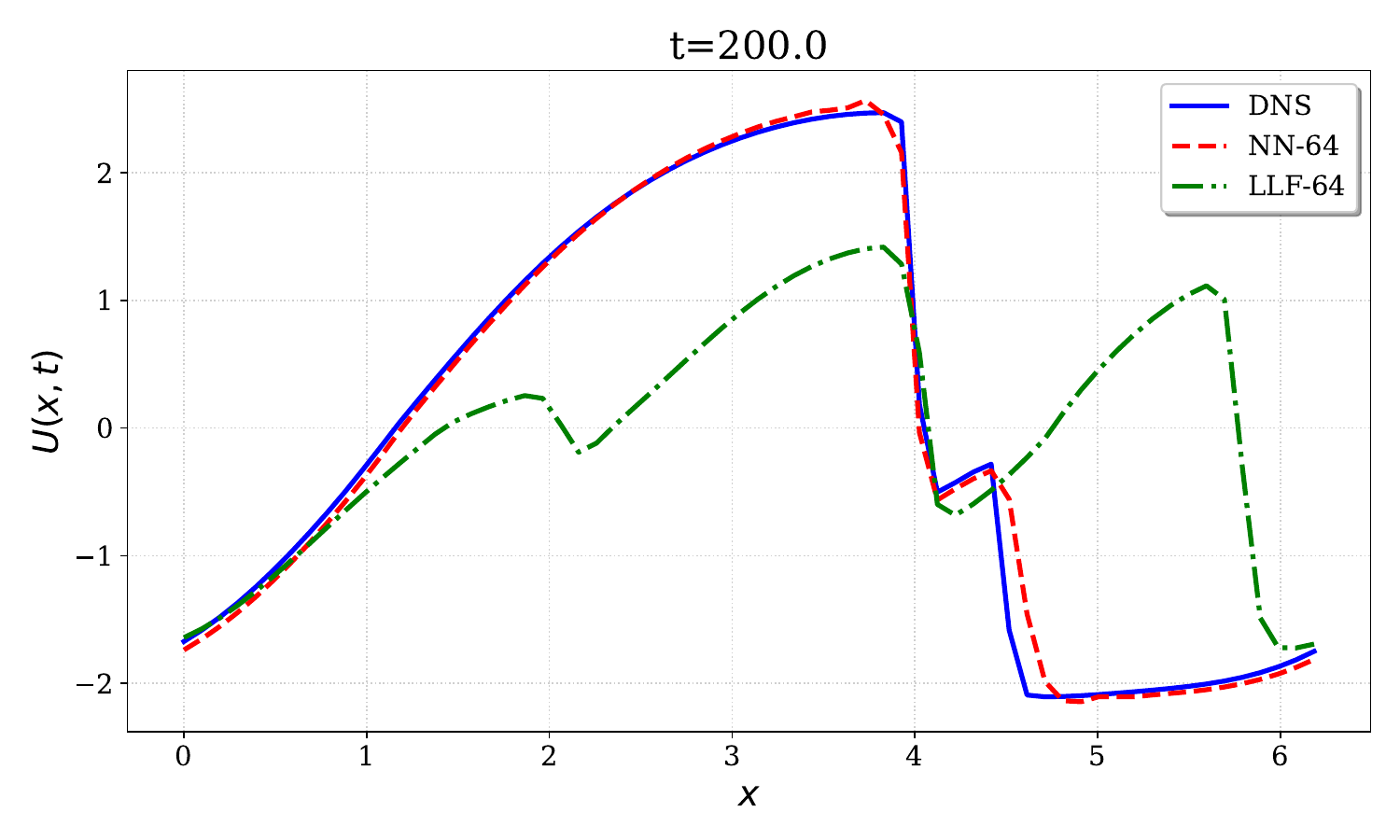}}
    \caption{Comparison between velocity snapshots $U(x,t)$ for the Burgers' equation in simulations of the NN-64 reduced model (red dashed), \rev{LLF-64 (green dash-dot),} and DNS-512 (blue solid) at times $t = 8, 50, 100, 200$.}
    \label{fig:burgers_snapshots}
\end{figure}

%%%%%%%%%%%%%%%%%%%
% Extrapolation results
%%%%%%%%%%%%%%%%%%%
\subsection{Extrapolation Results}
\label{sec:generalize}

One important practical concern regarding neural networks is their ability to generalize beyond the training regime. 
To evaluate this, we apply the NN reduced model without \revv{retraining} to different forcing regimes and initial conditions, analyzing its robustness and adaptability.

\textbf{Larger Forcing Regimes.}
To assess the robustness and physical fidelity of our framework, we perform simulations of the reduced model with out-of-distribution forcing. In particular, the forcing amplitude is increased from the baseline training value of $A = 1.0$ to 
$A = 1.2$ and $A = 1.4$, representing increases of 20\% and 40\% in the energy injected into the system, respectively. 
Recall that forcing is given by \eqref{eq:force}.
The forcing is injected into the larger scales and 
cascades down to the unresolved subgrid scales through nonlinear interactions. Accurate subgrid parameterization must dynamically adapt to this increased energy throughput to maintain stability and spectral accuracy.

The energy spectra for simulations with larger forcing are presented in Figure \ref{fig:spectra_forcing_increase}. 
The NN-64 reduced model accurately tracks the DNS energy cascade across the entire resolved wavenumber range. The model achieves this consistency across all forcing regimes without retraining. The localized Eddy Viscosity network, bounded by physical constraints, successfully adjusts the subgrid dissipation dynamically. It applies the precise amount of dissipation required to prevent spurious oscillations while preserving the high-wavenumber tail. 
We also verified that solution snapshots are reproduced very well in regimes with increased forcing. Figure \ref{fig:burgers_snapshots1.4} depicts snapshots in simulations with $A=1.4$. The same initial condition and forcing realization are used in simulations with $A=1$ depicted in Figure \ref{fig:burgers_snapshots}.

As expected, the overall magnitude of the solution increases with larger forcing, but all networks generalize well in this regime.
In particular, the Eddy Viscosity network suppresses spurious oscillations near shocks, while the Structure-Preserving network models the nonlinear wave interactions.
Our numerical results confirm that all three networks successfully learned the physics of the nonlinear energy cascade, allowing them to generalize robustly beyond the original training regime.

\begin{figure}[htbp]
    \centerline{
    % A = 1.2 (20% increase)
    \includegraphics[width=0.5\linewidth]{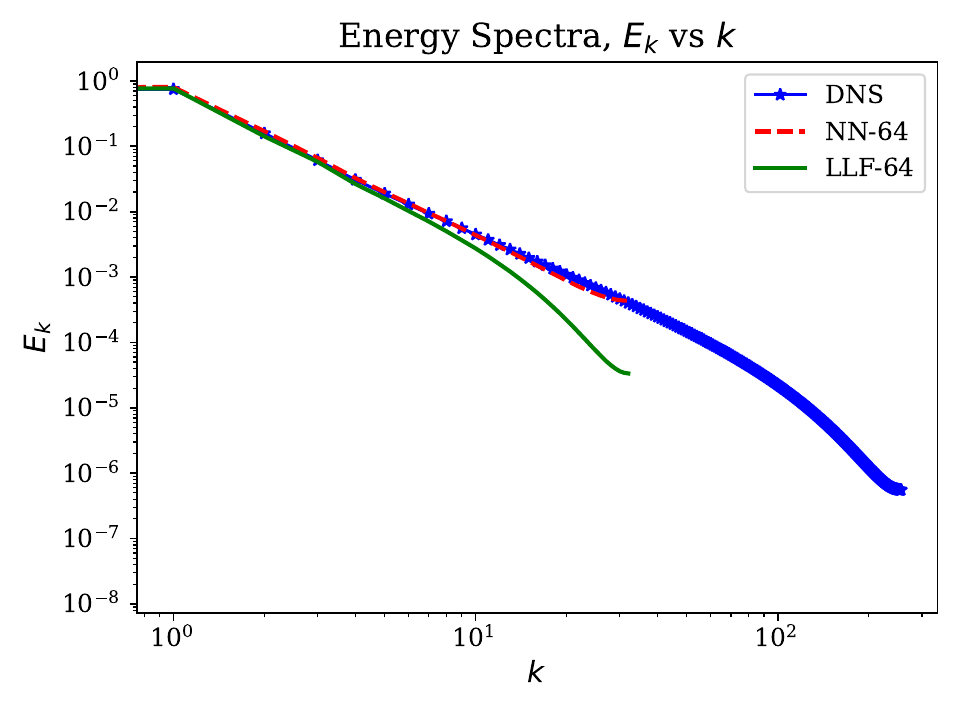}
    % A = 1.4 (40% increase)
    \includegraphics[width=0.5\linewidth]{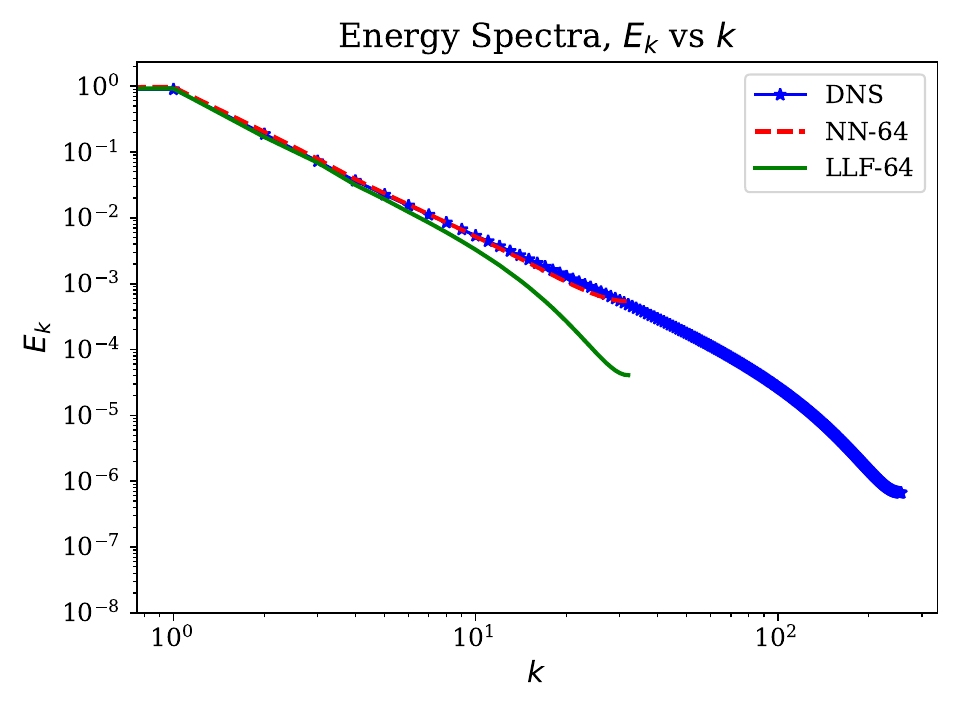}}
    \caption{Energy spectra of DNS, NN-64, and LLF-64 models with increased forcing $A=1.2$ (left) and $A=1.4$ (right). The NN-64 reduced model is simulated without retraining the NN parametrization.}
    \label{fig:spectra_forcing_increase}
\end{figure}

\begin{figure}[H]
    % t = 8, 20
    \centerline{
    \includegraphics[width=0.5\linewidth]{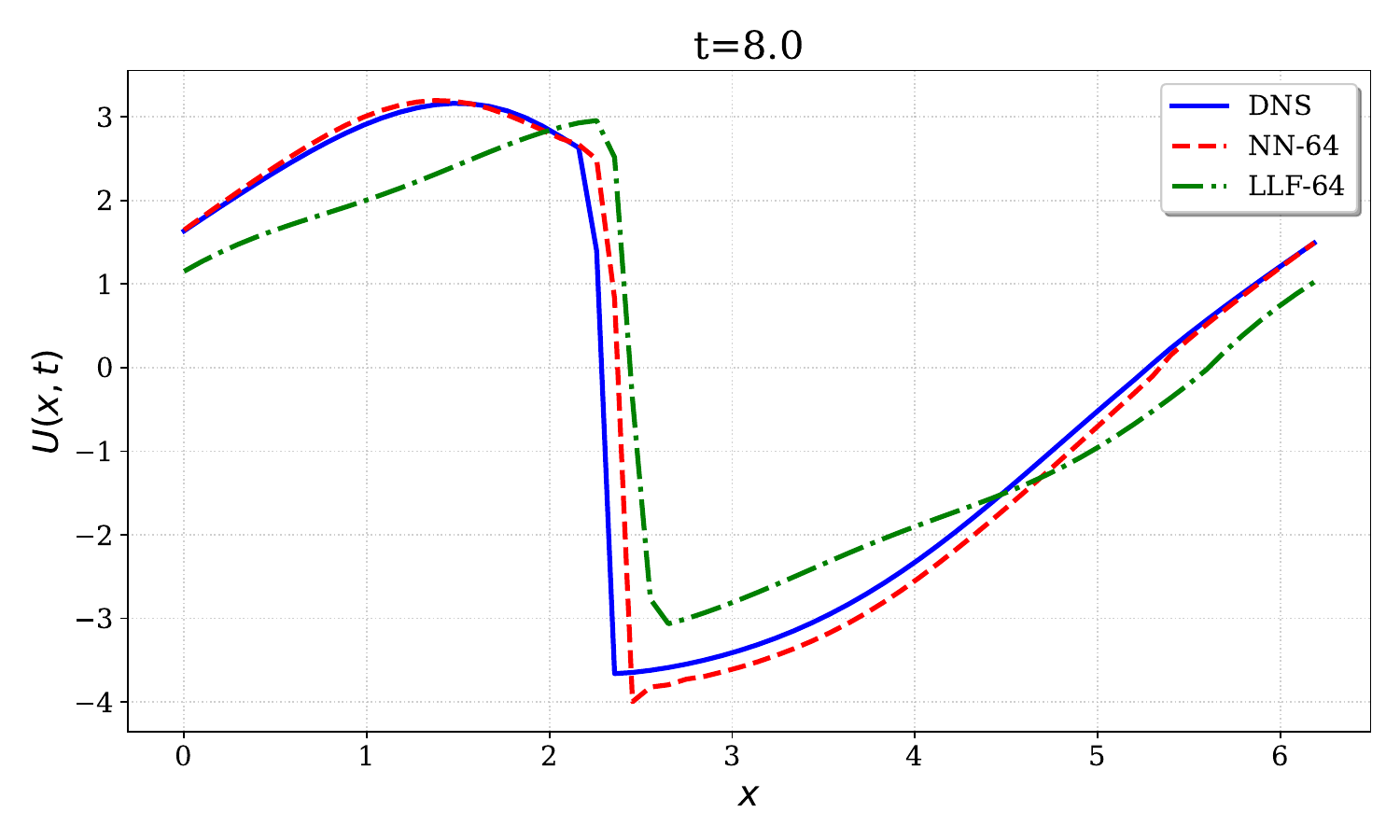} 
    \includegraphics[width=0.5\linewidth]{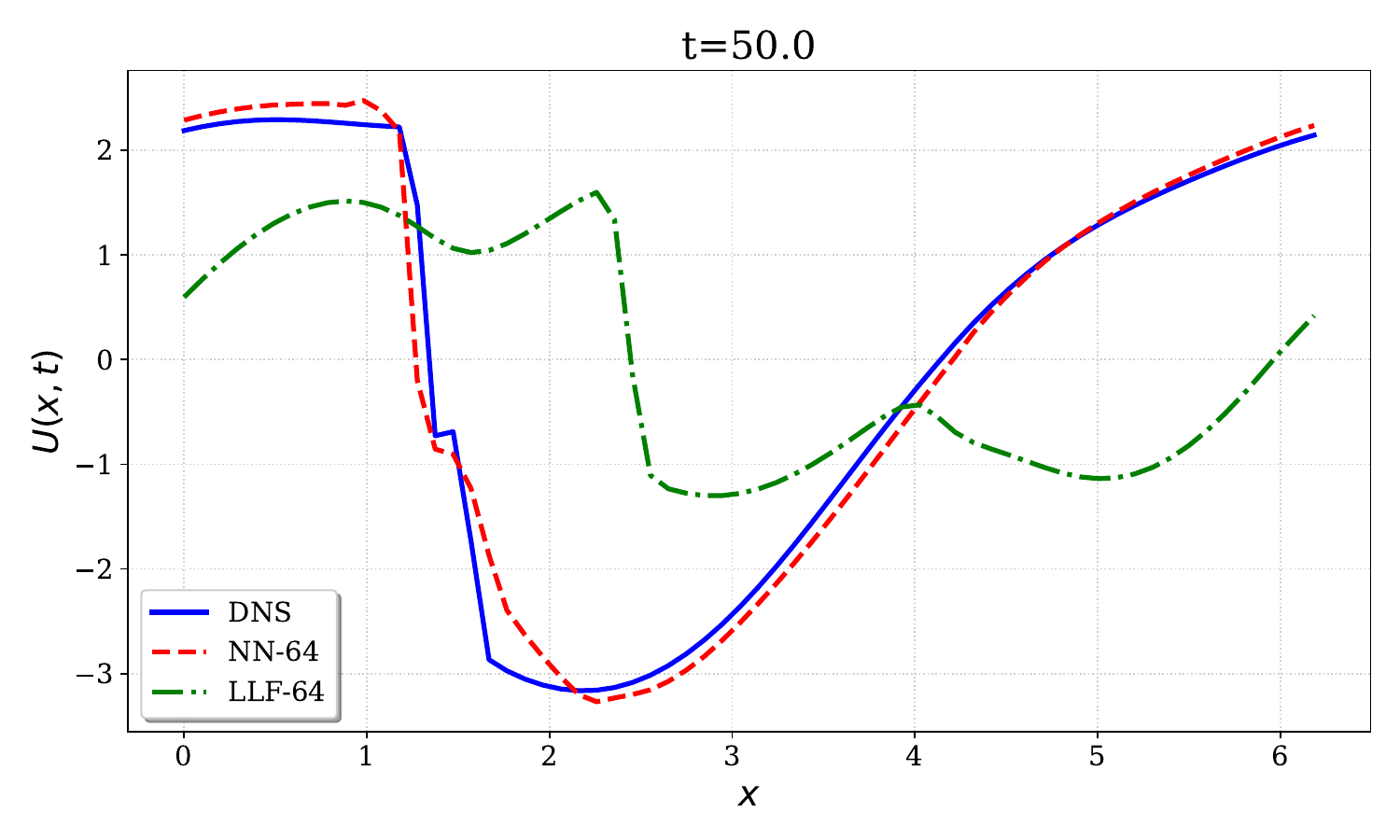}}
    % t = 100, 200
    \centerline{
    \includegraphics[width=0.5\linewidth]{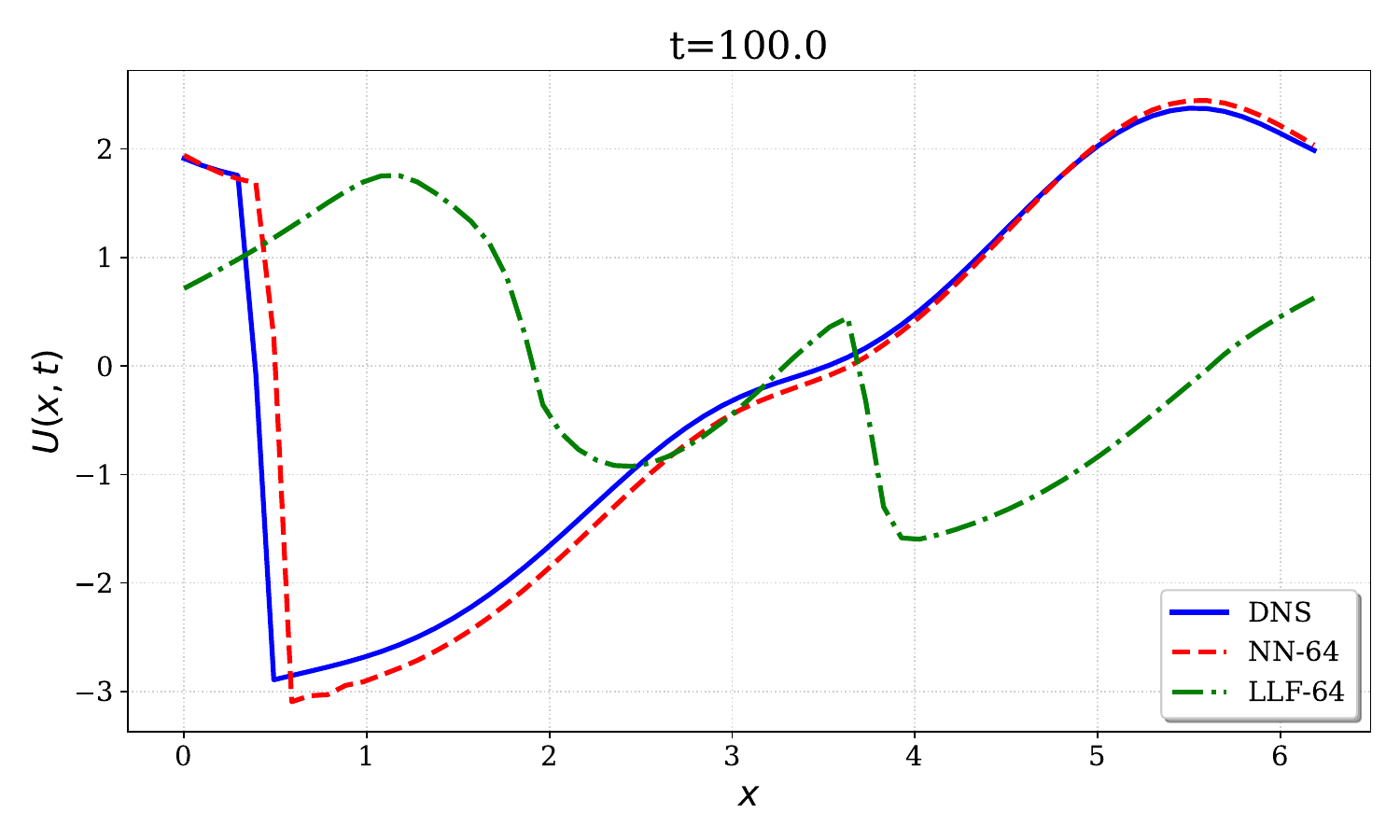} %
    \includegraphics[width=0.5\linewidth]{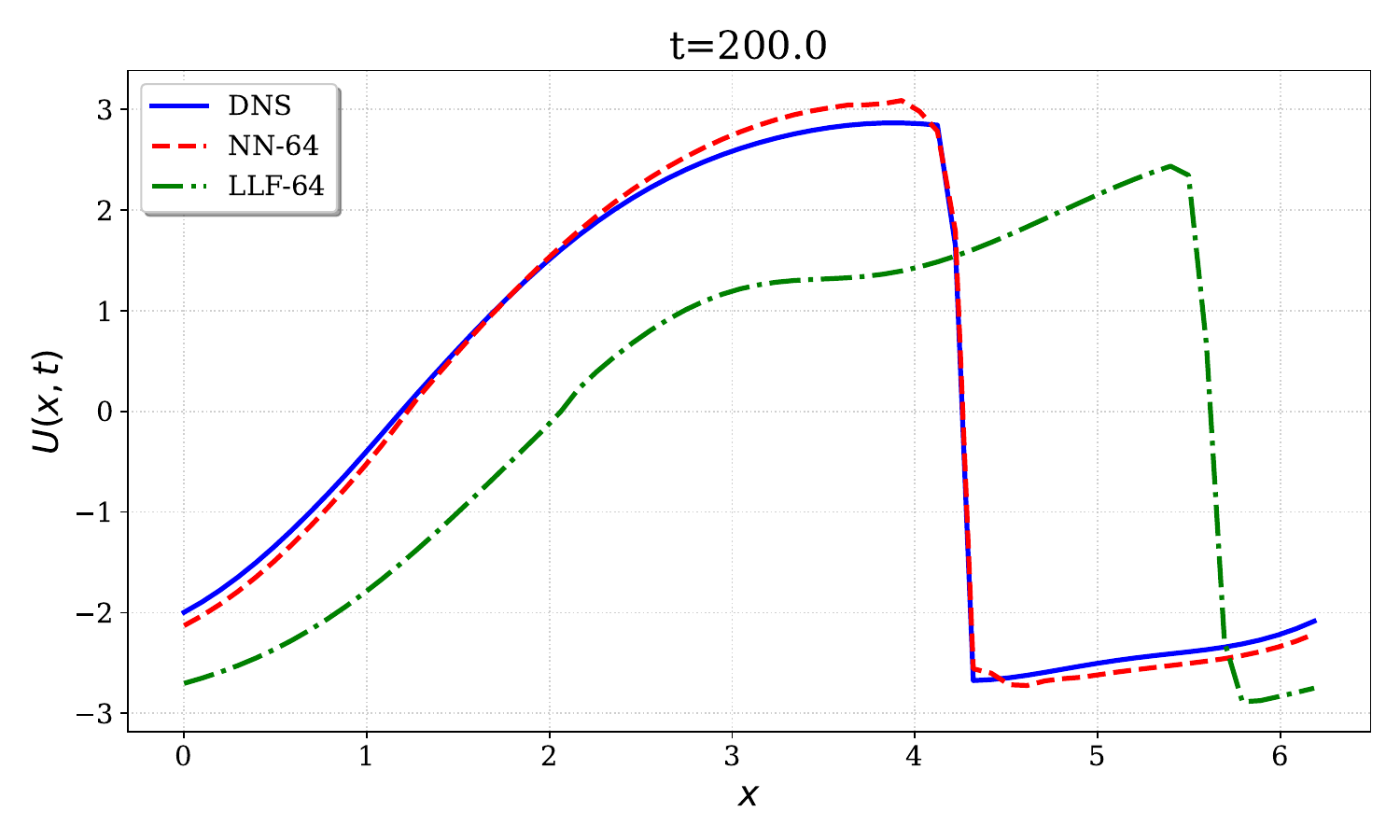}}
    \caption{Comparison between velocity snapshots $U(x,t)$ for the Burgers' equation in simulations of the NN-64 reduced model (red dashed), \rev{LLF-64 (green dash-dot),} and DNS-512 (blue solid) with increased forcing $A=1.4$ at times $t = 8, 50, 100, 200$. Compare with simulations with $A=1$ in Figure \ref{fig:burgers_snapshots}.
    Note that the NN parameterization was trained on the baseline regime $A = 1$ and is evaluated here without retraining.}
    \label{fig:burgers_snapshots1.4}
\end{figure}

\textbf{Step Initial Conditions.}
To further evaluate how well our framework generalizes to out-of-sample regimes, we consider discontinuous initial conditions. While the Neural Network parameterization was trained exclusively on smooth initial wave profiles, we evaluate its performance here on a discontinuous step function (a Riemann-type problem) without any retraining. 
The specific initial condition considered here is
\begin{equation*}
    u(x, 0) = 
    \begin{cases} 
      0.5, & \frac{2\pi}{3} \le x < \frac{4\pi}{3} \\ 
     -0.5, & \text{otherwise} 
    \end{cases}
\end{equation*}
Handling sharp discontinuities can be particularly challenging for machine learning parametrizations. Black-box models trained exclusively on smooth data typically suffer from severe, non-physical Gibbs oscillations or instability when encountering shocks (see e.g., \cite{coutinho2023physics, fuks2020limitations, krishnapriyan2021characterizing, mojamder2026subgrid, timofeyev2026subgrid}). 
In contrast, our structure-preserving NN parametrization remains remarkably stable and does not require any adjuSMRents for discontinuous data.

\revv{Figure \ref{fig:burgers_trajectories_step} compares the DNS, NN-64, and LLF-64 solution snapshots for the discontinuous initial condition at $t=0.1, 10, 40, 50$.}
At early simulation times (e.g., $t = 0.1$), the reduced-order model accurately captures the sharp, near-vertical gradients of the initial step without overshoots. As the simulation evolves ($t = 10, 40, 50$), the discontinuities propagate, interact, and decay. Throughout this evolution, the NN-64 reduced model consistently reproduces the correct shock speed and amplitude of the high-resolution DNS reference. 

This successful extrapolation highlights the robustness of our structure-preserving NN parametrization approach and the 
embedded Eddy Viscosity network in particular.
Because the architecture enforces physical bounds and localized dissipation dynamically, this network successfully recognizes and stabilizes steep gradients even in a physical regime absent from the training data.
\rev{Overall, the NN-64 parametrization provides a substantial improvement not only for averaged quantities but also for individual solution trajectories.}

\begin{figure}[H]
    % t = 0.1, 10
    \centerline{
    \includegraphics[width=0.5\linewidth]{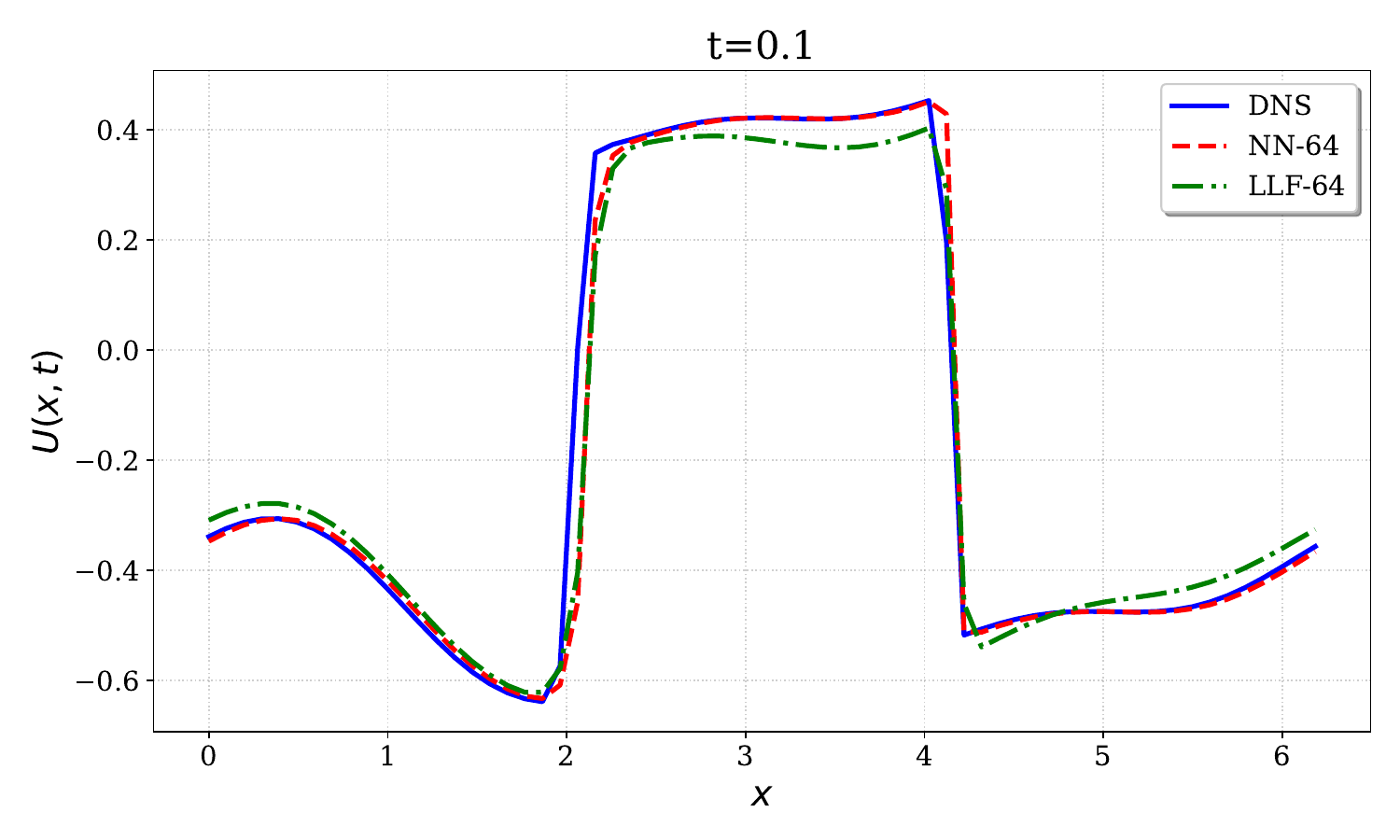} % 
    \includegraphics[width=0.5\linewidth]{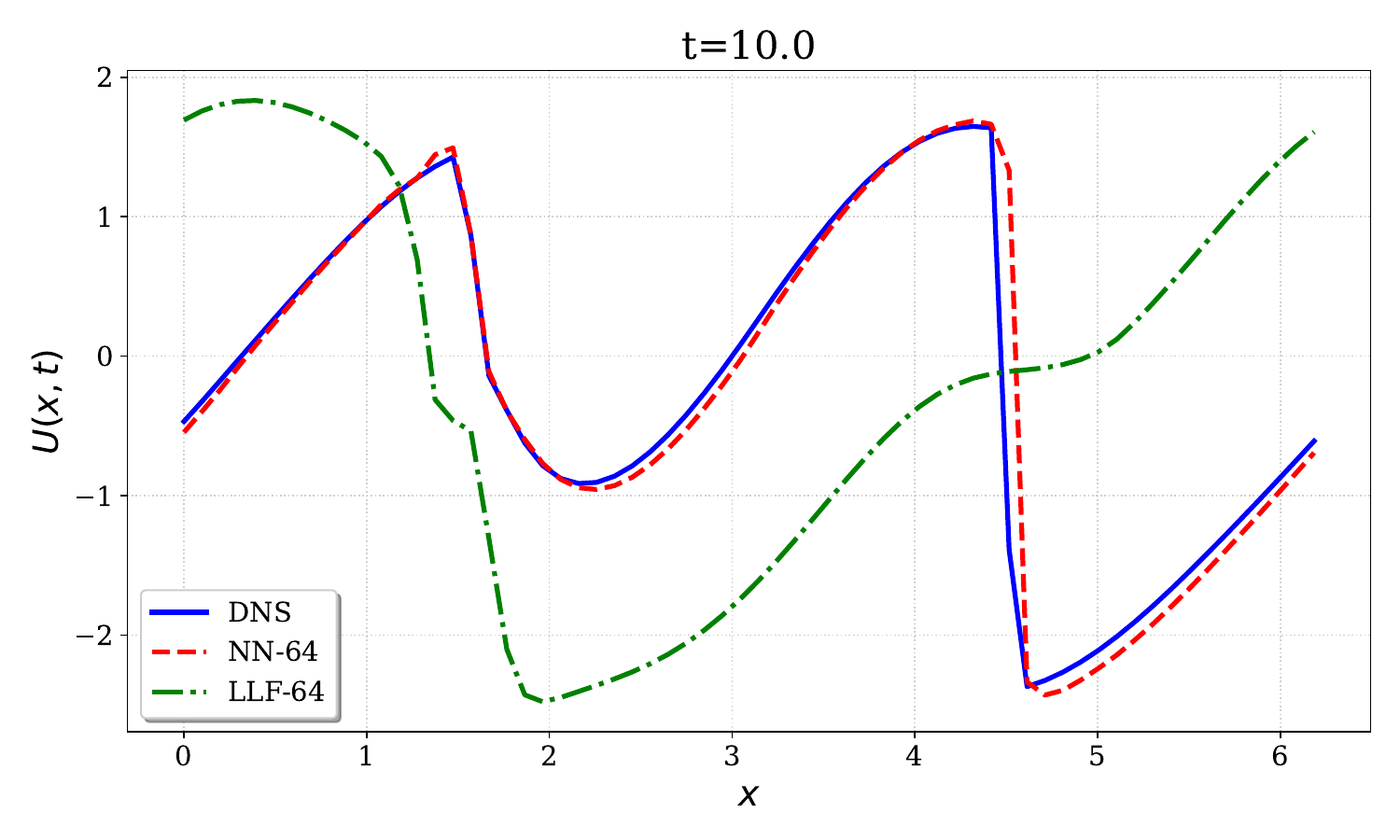}}
    % t = 40, 50
    \centerline{
    \includegraphics[width=0.5\linewidth]{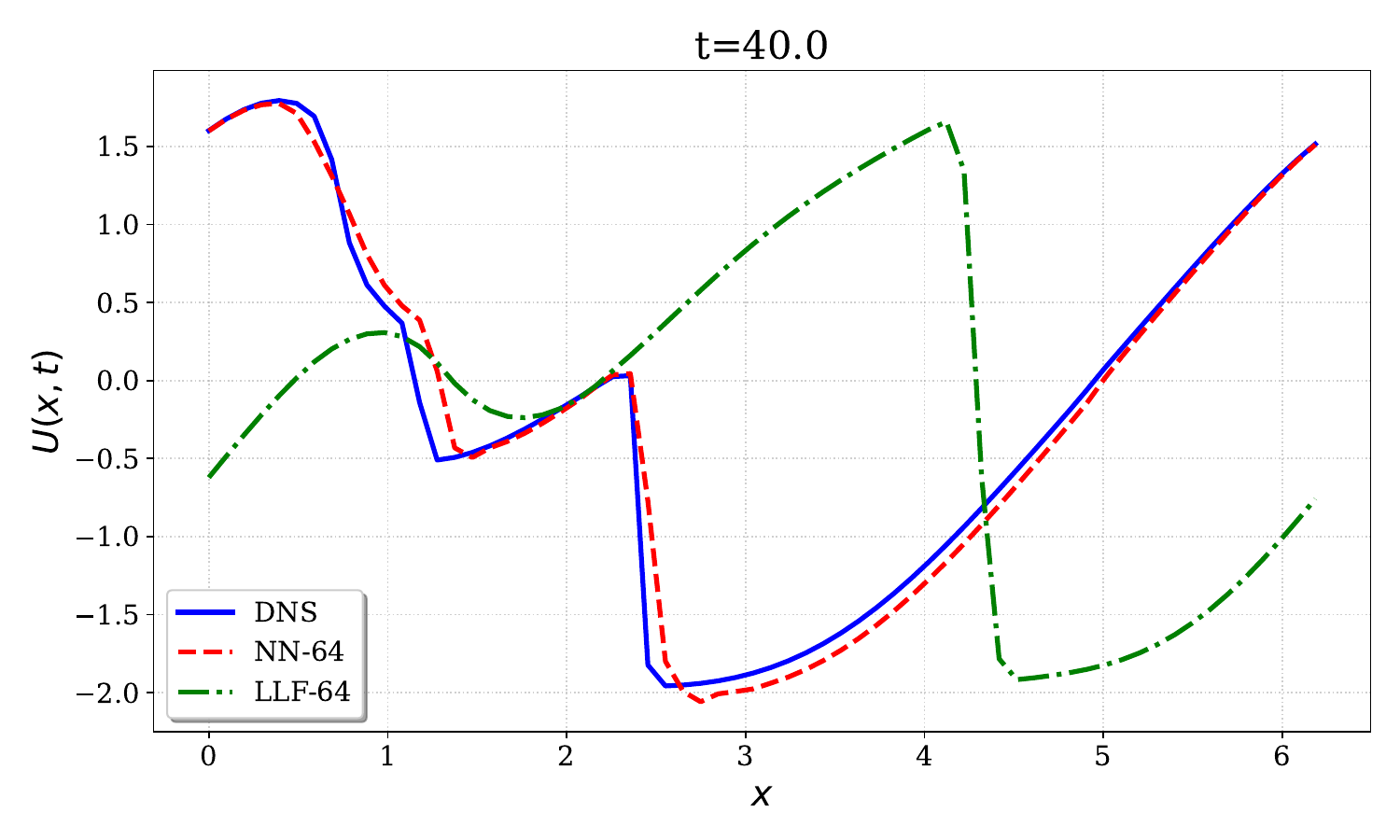} % 
    \includegraphics[width=0.5\linewidth]{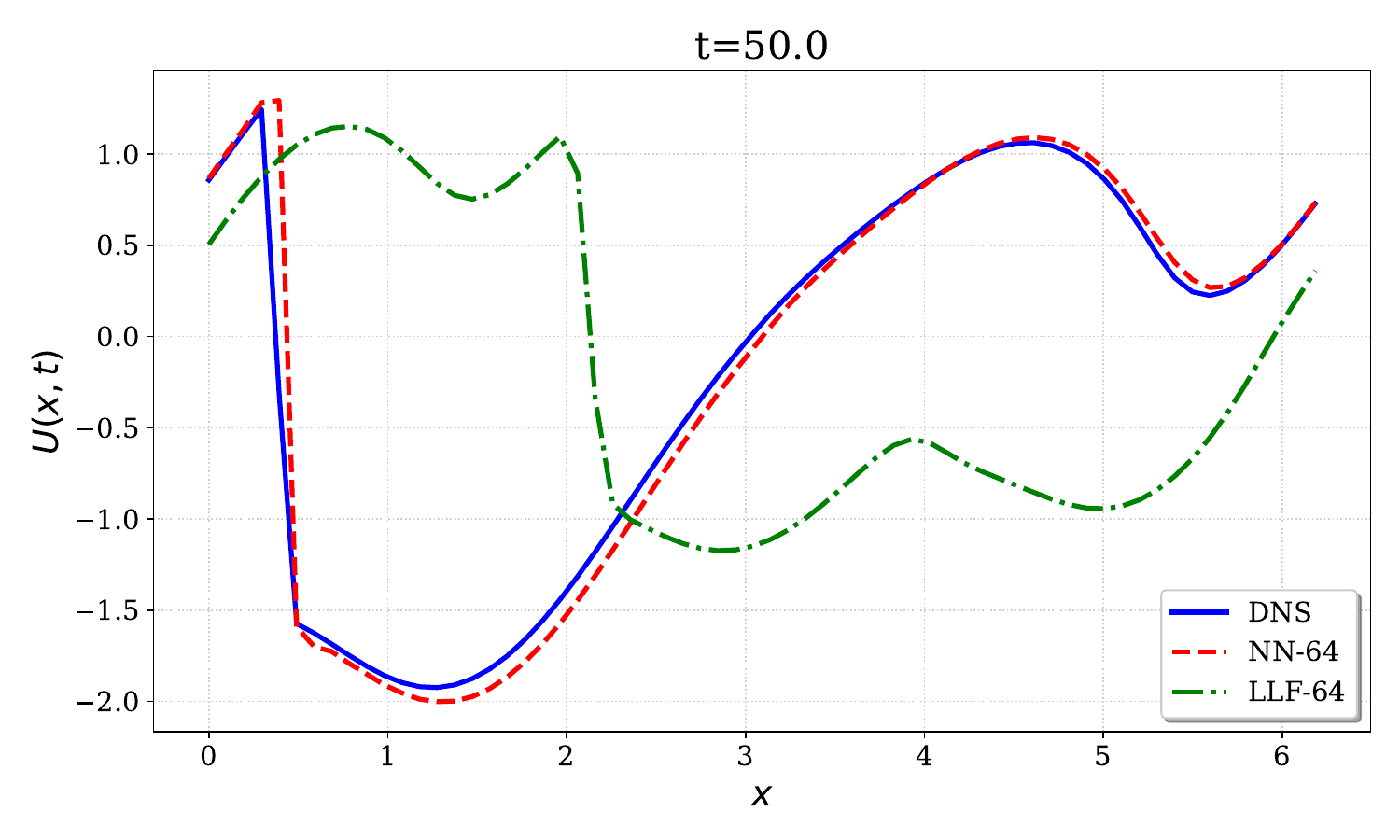}}
    \caption{
    Comparison between velocity snapshots $U(x,t)$ for the Burgers' equation in simulations of the NN-64 reduced model (red dashed), \rev{LLF-64 (green dash-dot),} and DNS-512 (blue solid) with step initial conditions at times $t = 0.1, 10, 40, 50$.
    Note that the NN parameterization was trained on solutions with smooth initial conditions and is evaluated here without retraining.}
    \label{fig:burgers_trajectories_step}
\end{figure}

\begin{revvblock}
The training time for the neural network model is approximately 12 minutes.
This fast training time reflects the design (relatively small depth and width) of all three neural networks.

Table \ref{tab:comp_cost} reports the inference cost for the fully resolved DNS model with $N_f=512$, the coarse LLF-64 and NN-64 models using the same time step as the DNS, and the NN-64 model using a time step four times larger, $4\Delta t$.
The inference time of the coarse NN-64 model using the smaller time step is greater than that of the DNS because evaluating the neural networks requires additional operations. However, the coarser spatial resolution allows the use of a larger time step, which substantially reduces the inference cost of the NN-64 model. We verified that simulations performed with the larger time step reproduce the same statistical quantities and trajectory predictions. We expect the relative computational savings to become more significant for multidimensional problems, where the cost of fully resolved simulations increases rapidly with spatial resolution.
\begin{table}[htbp]
\centering
\caption{Inference wall-clock time for simulation with $T=1000$ for the DNS ($N_f=512$) and simulations on a uniform coarse grid ($N_c=64$).}
\label{tab:comp_cost}
\begin{tabular}{lcc}
\toprule
\textbf{Scheme} & \textbf{Inference Time (s)} \\
\midrule
DNS & $243.46$     \\
%TVD-64 & $204.71$  \\
%Smagorinsky Static & $202.74$  \\
%Smagorinsky Dynamic & $573.14$    \\
LLF-64 & $165.59$ \\
NN-64 & $329.90$  \\
NN-64 with $4\Delta t$ & $80.87$ \\
\bottomrule
\end{tabular}
\end{table}
\end{revvblock}

%%%%%%%%%%%%%%%%%%%
% Conclusions
%%%%%%%%%%%%%%%%%%%
\section{Conclusions}
\label{sec:conc}
In this paper, we develop a structure-preserving neural network (NN) approach for parametrizing subgrid physical processes in coarse simulations of hyperbolic conservation laws. This structure-preserving approach relies on the underlying analytical properties of hyperbolic systems, but does not require explicit knowledge of the entropy function. Instead, the entropy function is learned using an input-convex neural network. Subgrid fluxes are then estimated using another network to learn the flux potential. A third neural network approximates the Eddy Viscosity and has design features similar to traditional eddy viscosity models.
\rev{The main contribution of this work is the formulation and numerical validation of a structure-preserving neural-network framework for subgrid closure. The novelty lies in adapting this methodology to represent unresolved fluxes in a prescribed coarse discretization and in combining the learned entropy function and flux potential with a neural-network eddy-viscosity closure.}

We demonstrate that our NN reduced model performs very well and reproduces the statistical properties of solutions with high accuracy. In particular, there is no over-damping at higher wavenumbers in the energy spectra of coarse variables $U_I$ (see Figure \ref{fig:energy_spectra_base}), which is usually difficult to achieve.
We also demonstrate that the NN parametrization is structurally stable and that the solutions of the NN reduced model do not develop spurious oscillations.
Consequently, the proposed NN parametrization approach does not require additional stabilization techniques.

Numerical results in section \ref{sec:generalize} show that the NN reduced model generalizes well outside the training regime.  In simulations with larger forcing, both the total energy and the magnitude of solutions increase, and the NN parametrization operates outside of the training regime. The NN reduced model reproduces the energy spectra and individual solutions in simulations with 
increased forcing very well.
Numerical simulations with discontinuous initial conditions 
also demonstrate that the NN reduced model 
generalizes well to regimes with discontinuous initial data.
Results with larger forcing are particularly relevant for potential applications in atmospheric fluid dynamics, where it is important to consider increased atmospheric forcing on the ocean due to climate change scenarios.

\rev{The numerical comparisons with the static and dynamic Smagorinsky models further demonstrate the advantages of the proposed formulation. For the Burgers-equation test cases considered here, which include solutions with shocks, the static model introduces excessive dissipation, whereas the dynamic model can produce excessively large stresses because its scale-invariance assumption is not well suited to these cases. The comparison with the second-order TVD method also provides a conventional shock-capturing baseline. These results show that the proposed closure improves the representation of unresolved dynamics relative to the classical models considered in this study.}

\rev{The analytical results in Section \ref{sec:analyt} demonstrate that the resulting NN parametrization is entropy stable and provide useful guidelines for selecting the range for the diffusion constant in the Eddy Viscosity Network. This highlights the advantage of a structure-preserving approach to data-driven flux estimation in the reduced model. In contrast to black-box closure strategies, the proposed formulation preserves key stability properties while retaining the flexibility of neural-network-based parametrizations. This combination of analytical tractability, numerical robustness, and improved predictive accuracy makes the approach particularly attractive for reduced-order modeling of hyperbolic conservation laws.}

\rev{Overall, the NN parametrization yields a considerable improvement in the accuracy of both the energy spectra and individual solution trajectories. The model also extrapolates well to regimes with stronger forcing, suggesting that the learned closure captures robust coarse-scale dynamics rather than merely interpolating within the training distribution. A further advantage is that accurate performance is obtained using relatively small neural networks, making the approach computationally lightweight and straightforward to train compared, for instance, with WGAN-based closure models. Finally, because the parametrization is formulated at the level of the fluxes, it remains amenable to analytical investigation, providing a useful bridge between data-driven modeling and structure-preserving numerical analysis.}

The structure-preserving approach allows the use of a small number of hidden layers and neurons in each layer. For instance, the Eddy Viscosity network has only one hidden layer.
In addition, we expect that our approach can be extended to other hyperbolic problems, but applications to more complex equations, and especially to systems of equations, introduce additional network-design considerations.
For instance, it is important to properly design the Eddy Viscosity network to avoid over-damping at higher wavenumbers of the coarse variables, which is a common problem in many closure models.
Therefore, it is necessary to select an appropriate vector of feature variables
$\bm{\xi}$ in \eqref{DNN}, and possibly increase the number of layers or split the Eddy Viscosity network into several networks with scalar outputs to avoid inaccurate estimation of the vector $C(\bm{\xi})$.
This requires a separate investigation and will be addressed in future work.

\rev{We do not foresee conceptual issues in extending our approach to systems of equations and multidimensional problems. Therefore, extending the structure-preserving framework to systems of conservation laws and multidimensional problems represents a natural and highly promising direction for future work. We expect that the overall architecture of the Entropy Neural Network and the Flux
Potential Neural Network would remain similar. However, it might be necessary to increase the number of layers and/or the number of neurons to incorporate multidimensional spatial stencils. 
Developing an efficient and accurate Eddy Viscosity network might require a separate investigation, as discussed in the previous paragraph.
For multidimensional systems, the entropy function would depend on the full vector of state variables, and the neural-network architecture would need to account for fluxes in multiple spatial directions. The Eddy Viscosity Network would also need to incorporate multidimensional spatial derivatives and interactions between different directions.}
\revv{The resulting increase in the number of state variables, spatial stencils, and model parameters may also lead to more demanding and potentially less well-conditioned optimization problems. Determining suitable architectures and training strategies that preserve entropy stability and provide adequate dissipation without introducing excessive numerical diffusion will therefore require a dedicated investigation. Moreover, extending the entropy-stability analysis would likely require additional structural assumptions on the neural networks, potentially imposing further constraints on their architecture.}

%=========================
\section*{Acknowledgments} The authors thank Dr. Lu Zhang for helpful discussions.

% ---------- Bibliography ----------
%\bibliography{refb}      % Assumes a file "refb.bib"

\end{document}